\documentclass[11pt,letterpaper,reqno]{amsart}
\usepackage[margin=1in]{geometry}
\usepackage{amsmath,amssymb,amsthm}
\usepackage[hidelinks]{hyperref}
\usepackage{comment}
\usepackage{aliascnt}
\usepackage[capitalise,noabbrev]{cleveref}
\usepackage{mathrsfs}
\usepackage{empheq}
\usepackage{xcolor}
\usepackage{enumitem}

\newenvironment{deferredproof}[1]
  {%
    \hypertarget{proof:#1}{}%
    \begin{proof}[Proof of \Cref{#1}]%
  }
  {%
    \end{proof}%
  }

\newcommand{\proofref}[2]{%
  \hyperlink{proof:#1}{the proof in \Cref*{#2}}%
}

\numberwithin{equation}{section}

\theoremstyle{plain}
\newtheorem{theorem}{Theorem}[section]

\newaliascnt{proposition}{theorem}
\newtheorem{proposition}[proposition]{Proposition}
\aliascntresetthe{proposition}

\newaliascnt{lemma}{theorem}
\newtheorem{lemma}[lemma]{Lemma}
\aliascntresetthe{lemma}

\newaliascnt{corollary}{theorem}

\aliascntresetthe{corollary}

\theoremstyle{definition}

\newaliascnt{definition}{theorem}

\aliascntresetthe{definition}

\newaliascnt{example}{theorem}

\aliascntresetthe{example}

\theoremstyle{remark}

\newaliascnt{remark}{theorem}
\newtheorem{remark}[remark]{Remark}
\aliascntresetthe{remark}

\Crefname{theorem}{Theorem}{Theorems}
\crefname{theorem}{theorem}{theorems}

\Crefname{proposition}{Proposition}{Propositions}
\crefname{proposition}{proposition}{propositions}

\Crefname{lemma}{Lemma}{Lemmas}
\crefname{lemma}{lemma}{lemmas}

\Crefname{corollary}{Corollary}{Corollaries}
\crefname{corollary}{corollary}{corollaries}

\Crefname{definition}{Definition}{Definitions}
\crefname{definition}{definition}{definitions}

\Crefname{example}{Example}{Examples}
\crefname{example}{example}{examples}

\Crefname{remark}{Remark}{Remarks}
\crefname{remark}{remark}{remarks}

\DeclareMathOperator{\diag}{diag}

\DeclareMathOperator{\id}{Id}

\newcommand{\Crit}{\operatorname{Crit}}
\newcommand{\Leb}{\operatorname{Leb}}

\title[Uniqueness of parabolic rays]{Strong Stable Manifolds and Uniqueness of Parabolic Rays in the Newtonian \(N\)-Body Problem}

\author{Putian Yang}
\address{School of Mathematical and Physical Sciences, Chongqing University
of Science and Technology}
\email{PutianYang@outlook.com}

\author{Shiqing Zhang\textsuperscript{*}}
\address{School of Mathematics, Sichuan University}
\email{zhangshiqing@scu.edu.cn}

\thanks{\textsuperscript{*}Corresponding author.}

\begin{document}

\begin{abstract}
We study uniqueness of parabolic motions in the Newtonian $N$-body
problem using strong stable manifolds and direct action comparisons.
Fix a minimal normalized central configuration $c$, with no
nondegeneracy assumption. In McGehee coordinates, every parabolic
motion with this limiting configuration has a tail on the outgoing
strong stable manifold. Its local graph yields a smooth velocity field
on a truncated cone, selecting a collision-free parabolic motion that
stays in the cone and is unique among all such cone-confined motions.
For free-time minimizing rays, an explicit renormalized action excess
gives a no-exit estimate, hence uniqueness from a smaller cone without
a confinement assumption. Together with known existence results,
these cone theorems imply that every $c$-parabolic motion has a tail
that is a $c$-parabolic geodesic ray. Globally, we propagate the stable
graph backward and compare the excess on its regular branches. An
endpoint differential formula and Sard's theorem imply that the
configurations admitting more than one such minimizing ray form a
Lebesgue-null set. The uniqueness proofs combine local dynamics and
variational comparisons without using differentiability of a global
Hamilton--Jacobi solution.
\end{abstract}

\maketitle

\tableofcontents

\section{Introduction}
\label{sec:introduction}

Prescribing the asymptotic configuration of a gravitational system leads
to a boundary-value problem at infinity. In the completely parabolic
regime considered here, velocities tend to zero and the leading expansion has order
$t^{2/3}$. Such motions belong to the classical asymptotic theory of the
$N$-body problem \cite{Chazy1922,Saari1971,MarchalSaari1976}. Even when the limiting
configuration is fixed, an initial position need not determine a unique
motion: the initial velocity remains to be selected.

We study this selection problem for a fixed minimal normalized central
configuration. Our approach combines the strong stable manifold of an
outgoing McGehee equilibrium with an explicit renormalized action excess.
It gives a smooth velocity selection and uniqueness for all parabolic
motions confined to a suitable cone. Free-time minimality then yields
uniqueness from a smaller cone without an assumed confinement condition,
and uniqueness from almost every collision-free initial configuration.
The construction allows the prescribed central configuration to be
degenerate.

\subsection{Setting and main results}

Let $N\geq2$ bodies of positive masses $m_1,\ldots,m_N$ move in
$\mathbb R^d$, with $d\geq2$. We use the center-of-mass space
\[
 X=\left\{x=(x_1,\ldots,x_N):\ \sum_{i=1}^N m_i x_i=0\right\},
 \qquad m:=\dim X=d(N-1),
\]
with mass inner product $\langle x,y\rangle=\sum_i m_i x_i\cdot y_i$
and norm $\|x\|^2=\langle x,x\rangle$. Let $\Delta$ be the collision set
and $\Omega=X\setminus\Delta$. The potential, Lagrangian, and energy are
\begin{equation}
\label{eq:intro-lagrangian}
 U(x)=\sum_{i<j}\frac{m_i m_j}{|x_i-x_j|},\qquad
 L(x,\dot x)=\frac12\|\dot x\|^2+U(x),\qquad
 h(x,\dot x)=\frac12\|\dot x\|^2-U(x).
\end{equation}
Newton's equation is $\ddot x=\nabla U(x)$, with the gradient taken in
the mass metric. Set $\mathcal E=\{s\in X:\|s\|=1\}$ and
$\mathcal E_0=\mathcal E\cap\Omega$, and write $d_{\mathcal E}$ for
spherical distance. A normalized central configuration is a critical
point of $U|_{\mathcal E_0}$; its spherical gradient is
\[
 \nabla_{\mathcal E}U(s)
 =\nabla U(s)-\langle\nabla U(s),s\rangle s
 =\nabla U(s)+U(s)s.
\]
Throughout the paper, fix a \emph{minimal normalized central
configuration} (MNCC) $c$, and set
\begin{equation}
\label{eq:intro-minimal-configuration}
 U(c)=\min_{\mathcal E_0}U=:U_{min},\qquad
 v_*=\sqrt{2U_{min}},\qquad
 \alpha=\left(\frac{9U_{min}}{2}\right)^{1/3}.
\end{equation}
We fix the configuration $c$ itself, including its orientation. No
nondegeneracy or isolation assumption is imposed.

A \emph{$c$-parabolic motion} is a collision-free zero-energy Newtonian
solution $x(t)=r(t)s(t)$, defined for all $t\geq0$, with
$r(t)=\|x(t)\|\to\infty$ and $s(t)\to c$. These conditions imply
\begin{equation}
\label{eq:intro-parabolic-asymptotics}
 t^{-2/3}x(t)\longrightarrow\alpha c,
\end{equation}
as established in the McGehee correspondence below. A
\emph{$c$-parabolic geodesic ray} is such a motion whose every compact
subarc minimizes
$\mathcal A_{[a,b]}(x)=\int_a^b L(x,\dot x)\,dt$
among absolutely continuous curves with the same endpoints and arbitrary
positive travel time. Thus the minimizing condition is free-time
minimality. After reparametrization these are minimizing rays of the
zero-energy Jacobi--Maupertuis metric
$2U(x)\langle\cdot,\cdot\rangle$; dynamical statements below use physical
time.
Here \emph{minimal} refers to global minimization of $U$ on the
inertia sphere. Every such configuration is also \emph{minimizing} in
Maderna's variational sense, meaning that its parabolic homothetic
motion is a free-time minimizer \cite{Maderna2013}; the hypothesis used
below is the global potential-minimum condition
\eqref{eq:intro-minimal-configuration}.

For $R,\delta>0$, define the truncated configuration cone
\[
 \mathcal C_c(R,\delta)
 =\{x=rs\in\Omega:\ r>R,\ d_{\mathcal E}(s,c)<\delta\}.
\]

\begin{theorem}[Dynamical existence and cone-confined uniqueness]
\label{thm:uniqueness-cone-confined-parabolic-motion}
There exist $R_*,\delta_*>0$ and a $C^\infty$ map
$\mathcal V_c:\mathcal C_c(R_*,\delta_*)\to X$ such that, for every
$x_0\in\mathcal C_c(R_*,\delta_*)$, the solution with initial data
\[
 \gamma_{x_0}(0)=x_0,\qquad
 \dot\gamma_{x_0}(0)=\mathcal V_c(x_0)
\]
is a collision-free zero-energy $c$-parabolic motion and remains in
$\mathcal C_c(R_*,\delta_*)$ for all $t\geq0$. It is the unique
$c$-parabolic motion starting at $x_0$ whose entire forward
configuration remains in this cone. No minimizing assumption is required.
\end{theorem}

The selected graph is positively invariant over this same cone. A
separate estimate places every cone-confined parabolic competitor on
the graph. For minimizing rays, an action comparison supplies the
confinement condition from initial data alone.

\begin{theorem}[No exit and cone uniqueness for parabolic rays]
\label{thm:uniqueness-cone-parabolic-geodesic-ray}
After increasing $R_*$ if necessary, there exists
$\varepsilon\in(0,\delta_*)$ such that, for every
\begin{equation}
\label{eq:cone-theorem-data}
 x_0=r_0s_0\in\Omega,\qquad r_0>2R_*,\qquad
 d_{\mathcal E}(s_0,c)<\varepsilon,
\end{equation}
every $c$-parabolic geodesic ray starting at $x_0$ remains in
$\mathcal C_c(R_*,\delta_*)$. Consequently, there is at most one such
ray from $x_0$, without a confinement hypothesis in its definition.
\end{theorem}

\medskip
\noindent\emph{Eventual free-time minimality.}
Together with the known existence of $c$-parabolic geodesic rays from
every collision-free configuration
\cite{PercinoSanchezMorgado2014,PolimeniTerracini2024},
\Cref{thm:uniqueness-cone-confined-parabolic-motion,thm:uniqueness-cone-parabolic-geodesic-ray}
imply that every $c$-parabolic motion is eventually free-time minimizing.
More precisely, for every $c$-parabolic motion $x:[0,\infty)\to\Omega$,
there exists $T\geq0$ such that
\[
 x_T(t):=x(T+t),\qquad t\geq0,
\]
is a $c$-parabolic geodesic ray.
Indeed, choose $T$ so large that
$x([T,\infty))\subset\mathcal C_c(2R_*,\varepsilon)$, and let $\eta$
be a $c$-parabolic geodesic ray starting at $x(T)$.
\Cref{thm:uniqueness-cone-parabolic-geodesic-ray} keeps $\eta$ in
$\mathcal C_c(R_*,\delta_*)$, while $x_T$ already stays in this cone.
\Cref{thm:uniqueness-cone-confined-parabolic-motion} therefore gives
$x_T=\eta$. No minimality of the initial portion $x|_{[0,T]}$ is asserted.

For the planar three-body problem, eventual free-time minimality of
parabolic motions asymptotic to a Lagrange configuration was proved by
Moeckel--Montgomery--S\'anchez Morgado
\cite[Theorem~2]{MoeckelMontgomerySanchezMorgado2018}.
The argument above gives the corresponding conclusion for general $N$
and $d\geq2$, with a prescribed MNCC and without a nondegeneracy
assumption.

\begin{theorem}[Global almost-everywhere uniqueness]
\label{thm:global-generic-uniqueness}
Let
\begin{equation}
\label{eq:multiple-ray-set}
 \mathfrak M_c
 =\{x\in\Omega:\ \text{at least two distinct $c$-parabolic geodesic
 rays start at $x$}\}.
\end{equation}
Then $\Leb(\mathfrak M_c)=0$. Hence, for almost every $x\in\Omega$,
there is at most one $c$-parabolic geodesic ray starting at $x$.
\end{theorem}

Existence of a free-time minimizing $c$-parabolic ray from each prescribed
collision-free configuration follows from
\cite{PercinoSanchezMorgado2014,PolimeniTerracini2024}.
Together with these existence results, the last two theorems give
exactly one ray in the smaller cone and from almost every
collision-free configuration, respectively. In the smaller cone the ray
is the motion selected by $\mathcal V_c$. The masses and $c$ remain
fixed throughout; \emph{generic uniqueness} refers to this
measure-theoretic statement for initial configurations. The two
conclusions without a confinement hypothesis concern minimizing rays.

\subsection{Relation to earlier work}

The variational theory rests in part on Marchal's theorem excluding
interior collisions of fixed-endpoint action minimizers in dimension
$d\geq2$ \cite{Marchal2002}.
Maderna--Venturelli \cite{MadernaVenturelli2009} constructed parabolic
motions from an arbitrary initial configuration to a prescribed minimal
central configuration, minimizing the action on every compact interval
with its duration fixed, without a nondegeneracy assumption.
Maderna's weak KAM theory \cite{Maderna2012} provides a framework for
free-time minimization. Da Luz--Maderna \cite{daLuzMaderna2014} proved
that free-time minimizers on a half-line are completely parabolic, and
Percino--S\'anchez-Morgado \cite{PercinoSanchezMorgado2014} constructed
Busemann functions whose calibrated rays start from every configuration
and have the prescribed minimal configuration at infinity.
More generally, Burgos--Maderna \cite{BurgosMaderna2022} proved that
Jacobi--Maupertuis geodesic rays at nonnegative energy are expansive and
described their asymptotic cluster decomposition.

The dynamical background includes McGehee's blow-up
\cite{McGehee1974}, the analysis of total-collapse and completely
parabolic manifolds by Saari--Hulkower \cite{SaariHulkower1981}, and
Chenciner's treatment of the relation between these two regimes through
homogeneity reduction \cite{Chenciner1998}.
A particularly close precedent is
Moeckel--Montgomery--S\'anchez Morgado
\cite{MoeckelMontgomerySanchezMorgado2018}. For the planar three-body
problem, they represent the stable manifold of a Lagrange rest point at
infinity as a smooth Lagrangian graph over a truncated configuration
cone and establish the eventual free-time minimality recalled above.
Our analysis develops this stable-graph perspective for general $N$,
$d\geq2$, and a prescribed MNCC, allowing degeneracy beyond rotational
directions.

A related connection between invariant manifolds and minimality was
established by Barutello--Canneori--Terracini
\cite{BarutelloCanneoriTerracini2021} for finite-time collisions.
They identify the local stable manifold near a minimal central
configuration with the initial data of minimizing collision arcs.
Their hypotheses include nondegeneracy, formulated modulo rotations in
the $N$-body setting. This is an important counterpart of the
stable-graph approach, although the endpoint is a collision rather than
parabolic infinity.

For motions at infinity, Boscaggin--Dambrosio--Feltrin--Terracini
\cite{BoscagginDambrosioFeltrinTerracini2021} proved existence from
sufficiently distant initial configurations in a cone by a
functional-analytic argument. Their spectral condition holds at every
MNCC, including degenerate ones. Thus cone-based existence without
nondegeneracy is already available; our emphasis is on capture of all
$c$-parabolic tails and on the associated uniqueness and no-exit
properties.
Polimeni--Terracini \cite{PolimeniTerracini2024} introduced a direct
renormalized-action construction of minimal expansive solutions with
prescribed initial and asymptotic data, including the parabolic rays
used here. Recent preprints extend this variational framework to
homogeneous interactions and refine the asymptotic analysis
\cite{BertiPolimeniTerracini2026}, and construct families with prescribed
subleading asymptotic modes \cite{Wu2026}. Prescribing these additional
modes at infinity is different from fixing an initial position and the
leading limiting configuration, as in the present problem.

At positive energy, Maderna--Venturelli
\cite{MadernaVenturelli2020} developed a viscosity-solution approach to
hyperbolic motions with prescribed initial and asymptotic data.
Their subsequent work \cite{MadernaVenturelli2026} proves a cone theorem,
uniqueness of hyperbolic Busemann functions, eventual minimality of
hyperbolic motions, and almost-everywhere uniqueness of minimizing rays
with prescribed noncollision asymptotic velocity. These results provide
a close positive-energy counterpart of the questions considered here.

For the parabolic regime itself, Berti--Polimeni--Terracini
\cite{BertiPolimeniTerraciniRegularity2026} study Hamilton--Jacobi
solutions associated with renormalized value functions, including those
for a fixed MNCC without a nondegeneracy assumption. They obtain
countable codimension-one rectifiability of the singular set and finer
bounds for regular conjugate points. Their singularities encode
nonuniqueness of minimizers of the corresponding renormalized
functional, yielding almost-everywhere uniqueness and more detailed
exceptional-set information in that framework.

Against this background, the point of the present approach is the
combination of strong stable capture for every $c$-parabolic tail,
including motions not assumed minimizing, with explicit action
comparisons on finite-dimensional stable sheets. An invariant part of
the local graph gives dynamical existence and cone-confined uniqueness.
For minimizing rays, the renormalized excess prevents exit and separates
regular branches of backward stable sheets. This yields a proof of
almost-everywhere uniqueness without appealing to differentiability or
semiconcavity of a global Hamilton--Jacobi solution. The regularity analysis in
\cite{BertiPolimeniTerraciniRegularity2026} also uses Sard-type
estimates; here Sard's theorem is applied directly to the configuration
projections of countably many strong stable sheets. Our global
conclusion is the Lebesgue nullity of $\mathfrak M_c$.

\subsection{Outline of the argument}

Write $x=rs$ and introduce McGehee coordinates and time by
\begin{equation}
\label{eq:intro-mcgehee-variables}
 \rho=r^{-1/2},\qquad v=r^{1/2}\dot r,\qquad
 w=r^{3/2}\dot s\in T_s\mathcal E,\qquad
 \frac{dt}{d\tau}=r^{3/2}.
\end{equation}
The extended zero-energy phase space is
\begin{equation}
\label{eq:phase space of McGehee coordinates}
 \mathcal P=\left\{(\rho,s,v,w):\
 \rho\in\mathbb R,\ s\in\mathcal E_0,\ w\in T_s\mathcal E_0,\
 \frac12(v^2+\|w\|^2)=U(s)\right\}.
\end{equation}
It has dimension $2m-1$. Allowing signed $\rho$ extends the flow across
infinity; the physical region has $\rho>0$, and we write
$\mathcal P_0=\mathcal P\cap\{\rho\geq0\}$.
The $c$-parabolic motions correspond to physical orbits converging to
$e_c^+=(0,c,v_*,0)$.

\emph{Strong stable capture and the invariant cone.}
On $\mathcal P$, the radial equation is $v'=\frac12\|w\|^2$, where
prime denotes McGehee time. Along an orbit converging to $e_c^+$, the
defect $F=v_*-v\geq0$ satisfies
\begin{equation}
\label{eq:intro-capture}
 \|w\|^2=2(U(s)-U_{min})+(v_*+v)F,\qquad
 F'\leq-\frac12(v_*+v)F.
\end{equation}
Minimality of $c$ therefore gives exponential decay and, with the shape
equation, captures the full tail on the strong stable manifold.
Let $H=H_c:=\operatorname{Hess}_{\mathcal E}U(c)\geq0$.
For an eigenvalue $\mu$ of $H$, the stable shape eigenvalue is
$\lambda^-(\mu)=-v_*/4-\sqrt{v_*^2/16+\mu}\leq-v_*/2$.
Together with the radial eigenvalue $-v_*/2$, these eigenvalues give an
$m$-dimensional strong stable space whose projection onto $(\rho,s)$
is an isomorphism, even when $H$ has a kernel.

A Lyapunov--Perron construction and the inverse function theorem
give a smooth local graph $w=G(\rho,s)$. The construction lies in the
standard invariant-manifold framework \cite{HirschPughShub1977}.
With $V(s,w)=\sqrt{2U(s)-\|w\|^2}$, the associated physical velocity is
\begin{equation}
\label{eq:smooth-initial-velocity-map}
 \mathcal V_c(rs)=r^{-1/2}
 \left[V\bigl(s,G(r^{-1/2},s)\bigr)s+G(r^{-1/2},s)\right].
\end{equation}
The homothetic half-line gives $G(\rho,c)=0$. In normal coordinates
$s=\exp_c\xi$, put
$g(\rho,\xi)=[(d\exp_c)_\xi]^{-1}G(\rho,\exp_c\xi)$.
The stable spectrum implies, on a sufficiently small rectangle,
\begin{equation}
\label{eq:intro-positive-invariance}
 \langle\xi,g(\rho,\xi)\rangle\leq-\kappa\|\xi\|^2,\qquad
 \rho'\leq-\frac{v_*}{4}\rho
 \quad(\rho\geq0)
\end{equation}
for some $\kappa>0$. Hence this part of the graph is positively
invariant. A separate configuration-to-phase confinement estimate puts
every cone-confined $c$-parabolic competitor on the same graph.
These two steps prove
\Cref{thm:uniqueness-cone-confined-parabolic-motion}.

\emph{Action excess and no exit.}
Subtracting the homothetic radial primitive $B(r)=2v_*\sqrt r$ from the
action gives the nonnegative excess
\begin{equation}
\label{eq:intro-excess}
\begin{split}
 \mathcal R_{[a,b]}(x)
 &:=\mathcal A_{[a,b]}(x)-B(r(b))+B(r(a))\\
 &=\int_a^b\left[
 \frac12\left(\dot r-\frac{v_*}{\sqrt r}\right)^2
 +\frac12r^2\|\dot s\|^2+\frac{U(s)-U_{min}}r\right]dt.
\end{split}
\end{equation}
Along a zero-energy solution this is the endpoint difference
$[2\sqrt r\,(v-v_*)]_a^b$.
For a $c$-parabolic motion, $E=2\sqrt r\,(v_*-v)$ satisfies
$E'=-v_*E/2-2\sqrt r\,(U(s)-U_{min})\leq-v_*E/2$.
Consequently the total tail excess equals $E$ and tends to zero.
On a sufficiently late tail, where $v\geq0$, a remaining-length estimate gives
$d_{\mathcal E}(s,c)\leq4\sqrt{F/v_*}$, hence
$d_{\mathcal E}(s,c)^2\sqrt r\leq8E/v_*\to0$.
A shell adjustment through an angle $\theta$ at radius $r$ has
excess $O(\theta^2\sqrt r)$. Comparing a minimizing ray with
homothetic segments and such adjustments bounds its total excess by
$C\theta_0^2\sqrt{r_0}$, where
$\theta_0=d_{\mathcal E}(s_0,c)$.
Thus $F(0)\leq C\theta_0^2$. For small $\theta_0$, the radial velocity
stays positive, the radius increases, and a bound on the remaining
shape length keeps the ray in the larger cone. This proves the no-exit
statement in \Cref{thm:uniqueness-cone-parabolic-geodesic-ray}.

\emph{Backward sheets and global uniqueness.}
Backward propagation of the physical local strong stable manifold by
integer McGehee times covers all initial phase points under consideration
by countably many smooth $m$-dimensional sheets.
Sard's theorem makes the critical values of their configuration
projections negligible. On a regular branch
$Z_i(\rho,s)=(\rho,s,v_i,w_i)$, the excess and its differential are
\begin{equation}
\label{eq:intro-branch-excess}
 R_i=2\rho^{-1}(v_*-v_i),\qquad
 dR_i=2\rho^{-2}(v_i-v_*)\,d\rho
       -\rho^{-1}\langle w_i,ds\rangle.
\end{equation}
The formula follows directly from the linearized equations and decay
of tangent variations. Distinct phase points over the same
configuration have different branch differentials.
Free-time minimality and gluing at infinity with vanishing excess
show that two minimizing rays from the same configuration have equal
excess. Each equality between distinct regular branches therefore
lies locally on a smooth hypersurface. Countability, together with
Sard's theorem for the remaining critical values, proves
\Cref{thm:global-generic-uniqueness}.

\subsection{Organization}

Section~2 derives the McGehee equations and their correspondence with
parabolic motions. \Cref{secn:Linearization at infinity} studies the
linearization, and
\Cref{secn:strong stable manifold and its tangent space at ec+}
constructs the strong stable manifold.
\Cref{secn:cone-uniqueness} proves the graph representation, positive
invariance, and configuration-to-phase confinement, then develops
the excess estimates and proves the two cone theorems.
\Cref{sec:global-generic-uniqueness} establishes global
almost-everywhere uniqueness. The appendices contain the technical
proofs used in these arguments.

\section{McGehee flows}
In this section, we present the McGehee system together with its derivation.
The McGehee system studied in this paper is
\begin{empheq}[left=\empheqlbrace]{align}
        \rho' &= -\frac{1}{2}\rho v, \tag{M1}\label{eq:M1}\\
         s' &= w,  \tag{M2}\label{eq:M2}\\
         v' &= \frac12\|w\|^2,\tag{M3}\label{eq:M3}\\
         D_\tau w &= \nabla_\mathcal{E} U(s) - \frac{1}{2}vw. \tag{M4}\label{eq:M4}
\end{empheq}
We now show how it is derived. 
The derivative of the coordinates are written by 
$
\dot{ }
$ and $'$ as take the derivative of $t$ and $\tau$ respectively.
First, we compute the the derivative of $\rho$:
\begin{equation*}
    \rho' = r^{3/2} \dot{\rho} = r^{3/2} (-\dfrac{1}{2}r^{-3/2} \dot{r}) = -\frac{1}{2}\dot{r}=-\frac{1}{2}\rho v.
\end{equation*}

Second, using $w = r^{3/2}\dot{s}$ to obtain the equation for shape $s$:
\begin{equation*}
    s' = \frac{dt}{d\tau}\dot{s} =w.
\end{equation*}

Third, we derive the equation for  $v$.
 For any $s\in\mathcal{E}$ and $u\in T_s\mathcal{E}$, let $\gamma$ be a 
 curve  on $\mathcal{E}$ with $\gamma(0) = s$ and $\dot{\gamma}(0) = u$.
 Differentiating $\langle\gamma(\epsilon), \gamma(\epsilon)\rangle = 1$ at $\epsilon = 0$, we obtain $\langle u, s\rangle = 0$, thus $s$ is the normal direction
 of $\mathcal{E}$ in the ambient space $X$.
 On the otherhand, since $\langle s , \dot { s } \rangle = 0$ and then $\langle \ddot{s},s\rangle + \langle \dot{s},\dot{s}\rangle = 0$ for $s\in\mathcal{E}$. Thus 
 $\langle \ddot{s},s\rangle = -\|\dot{s}\|^2$ means $\ddot{s}$ has the component $ -\|\dot{s}\|^2$ on the normal $s$, hence the covariant derivative
 of $\dot{s} $ on $\mathcal{E}$ is 
 \begin{equation}\label{eq:tangential projection of w'}
        D_t\dot{s} = \ddot{s} + \|\dot{s}\|^2s.
 \end{equation}
 Now since
$$x = rs,$$ we compute
$$
\dot{x} = \dot{r}s + r\dot{s},$$  leading to the polar acceleration decomposition on $\mathcal{E}$
\begin{equation}\label{eq:polar acceleration decomposition}
        \begin{aligned}
            \ddot{x} &= \ddot{r}s + 2\dot{r}\dot{s} + r\ddot{s}.\\
        &=(\ddot{r}-r\|\dot{s}\|^2)s+(2\dot{r}\dot{s}+rD_t\dot{s}).
        \end{aligned}
\end{equation}
Therefore, the Newton equation $\ddot{x} = \nabla U(x) = \nabla U(rs)$ become  
\begin{equation}\label{eq:decomose Newton eqn}
        (\ddot{r}-r\|\dot{s}\|^2)s+(2\dot{r}\dot{s}+rD_t\dot{s}) = \nabla U(rs) = r^{-2} \nabla U(s).     
\end{equation}
Take inner product with $s$  and us Euler's formula to get
$$\ddot{r} - r \|\dot{s}\|^2 = r^{-2} \langle \nabla U(s), s \rangle = -r^{-2} U(s),$$
hence 
\begin{equation}\label{eq:eq2 for proving equiv of McGehee and Newton}
        r^2\ddot{r} = r^3\|\dot{s}\|^2 - U(s).
\end{equation}
On the otherhand,  we have 
\begin{equation*}
    \begin{aligned}
        v^{\prime} &=r^{3/2}\frac{d}{dt}(r^{1/2}\dot{r})\\
        &=r^{3/2}\left(\frac{1}{2}r^{-1/2}\dot{r}^2+r^{1/2}\ddot{r}\right)\\
        &=\frac{1}{2}r\dot{r}^2+r^2\ddot{r}.
    \end{aligned}
\end{equation*}
and
\begin{equation*}
        w=r^{3/2}\dot{s}\quad\Rightarrow\quad\|w\|^2=r^3\|\dot{s}\|^2.
\end{equation*}
Combining these equations with $v=r^{1/2}\dot{r}$ \eqref{eq:eq2 for proving equiv of McGehee and Newton} to have equation for radial coordinate:
\begin{equation*}
    v^{\prime}=\frac{1}{2}v^2+\|w\|^2-U(s).
\end{equation*}

Fourth, the equation for $w$ is formulated as follows.
Select the tangent component of \eqref{eq:decomose Newton eqn}:
$$2\dot{r}\dot{s} + rD_t\dot{s} = r^{-2}\nabla_{\mathcal{E}}U(s).$$
Hence 
\begin{equation}\label{tgt eqn of nabla U}
    2r^2\dot{r}\dot{s} + r^3D_t\dot{s} = \nabla_{\mathcal{E}}U(s).
\end{equation}
Taking the covariant derivative of $w = r^{3/2}\dot{s}$ with respect to $\tau$ on $\mathcal{E}$, which yields
\begin{equation}\label{eq:eq3 for proving equiv of McGehee and Newton}
    \begin{aligned}
        D_{\tau}w &= r^{3/2}D_t(r^{3/2}\dot{s})\\
             &= r^{3/2}\left(\frac{3}{2}r^{1/2}\dot{r}\dot{s} + r^{3/2}D_t\dot{s}\right)\\
             &= \frac{3}{2}r^2\dot{r}\dot{s} + r^3D_t\dot{s}.
    \end{aligned}
\end{equation}
Using \eqref{tgt eqn of nabla U}, we have
\begin{equation}\label{eq:eq4 for proving equiv of McGehee and Newton}
    \begin{aligned}
            D_{\tau}w &=\frac{3}{2}r^2\dot{r}\dot{s}+\nabla_\mathcal{E}U(s)-2r^2\dot{r}\dot{s}.\\
  &=\nabla_{\mathcal{E}}U(s)-\frac{1}{2}r^{2}\dot{r}\dot{s}.
    \end{aligned}
\end{equation}
Applying $r^{2}\dot{r}\dot{s}=(r^{1/2}\dot{r})(r^{3/2}\dot{s})=vw$, we obtain the equation
for shape coordinate $w$:
\begin{equation*}      
    D_\tau w=\nabla_\mathcal{E}U(s)-\frac{1}{2}v w
\end{equation*}

Last but not least, we need to find the energy relation of these coordinate.  At first, $\langle s , \dot {s} \rangle = 0$ and 
 $\dot {x} = \dot {r}s + r \dot{s}$ to have 
 \begin{equation}\label{eq:decomose |dot x|2}
        \| \dot { x } \| ^ { 2 } = \dot { r } ^ { 2 } + r ^ { 2 } \| \dot { s } \| ^ { 2 } ,
 \end{equation}
since 
$$
\dot { r } = r ^ { - 1 / 2 } v , \quad \dot { s } = r ^ { - 3 / 2 } w ,$$
so we have
$$\dot { x } = r ^ { - 1 / 2 } ( v s + w ) .$$
Because $s$ is  normal and $w\in T_s\mathcal{E}$,
we have 
$$
\| \dot{x} \|^2 = r^{-1}(v^2 + \| w \|^2).
$$
On the otherhand, the energy of parabolic motions reads:
$$\frac{1}{2} \| \dot{x} \|^2 - U(x)=0,$$
by 
$U(x) = r^{-1}U(s)$
we have 
$$
\frac{1}{2r}(v^2 + \| w \|^2) - \frac{1}{r}U(s) = 0,
$$
and finally the zero energy constraint of McGehee coordinates for parabolic motions
\begin{equation}\label{energy relation of McGehee coordinates}
    \frac{1}{2}(v^2 + \| w \|^2) = U(s),
\end{equation}
In particular, $v$ can be determined by $s$ and $w$ as 
\begin{equation*}
        v = \pm\sqrt{2U(s) - \|w\|^2},
\end{equation*}
We denote the radial velocity funciton
\begin{equation*}
        V(s,w) := \sqrt{2U(s) - \|w\|^2}.
\end{equation*}
Recall the  equation \eqref{eq:M3} of radial velocity $v$ and 
by the zero-energy constraint
\eqref{energy relation of McGehee coordinates}, therefore
\begin{equation}
v'=\frac12v^2+\|w\|^2-\frac12\left(v^2+\|w\|^2\right)
=
\frac12\|w\|^2\geq 0.
\end{equation}
We will find later that  
$v$ is determined once $s$ and $w$ are fixed, hence we sometimes write the reduced equilibrium 
$e_c^+ = ( 0, c, 0 )$ where the first and second $0$ represent the $\rho$ and $w$ respectively and $v$ is omitted.

Throughout this paper, the stable sets (manifolds) we discuss are exclusively those of system 
\eqref{eq:M1}\eqref{eq:M2}\eqref{eq:M3}\eqref{eq:M4}.
For simplicity, we occasionally write the McGehee system of these equations as 
\begin{equation*}
        Z' = \mathfrak{X}(Z),~~Z\in\mathcal{P}.
\end{equation*}
We deonte the flow of the system as $\Phi^\tau$.

Assume that $s\to c$ and $r(t)$ is bounded below by a positive constant, then $\tau\to+\infty$ is equivalent to $t\to+\infty$ and the smoothness of 
the McGehee vector field implies smoothnes of its physical counterpart $x(t)$. 
More precisely, let
\[
Y(\tau)=(\rho(\tau),s(\tau),v(\tau),w(\tau)),
\qquad \tau\in[\tau_0,+\infty),
\]
be a smooth McGehee orbit in the physical region \(\rho(\tau)>0\). 
 Set
\[
r(\tau)=\rho(\tau)^{-2}.
\]
Since $\dfrac{dt}{d\tau}=r^{3/2}>0$, \(t(\tau)\) is strictly increasing and has a smooth inverse on its image.
Let a zero-energy physical orbit be defined for all sufficiently
large $t$ and satisfy $s(t)\to c$ and $r(t)\to+\infty$.  Its McGehee time is
\[
\tau(t)=\tau(t_0)+\int_{t_0}^t r(u)^{-3/2}\,du.
\]
Since $s\to c$, $U(s(t))$ is bounded for large $t$.  The zero-energy relation gives
$|v(t)|\leq C$, and hence
\[
\frac{d}{dt}r(t)^{3/2}=\frac32v(t)\leq \frac32C.
\]
Therefore $r(t)^{3/2}\leq C_1(1+t)$ for large $t$, and
\[
\tau(t)-\tau(t_0)
\geq \frac1{C_1}\int_{t_0}^t\frac{du}{1+u}
\longrightarrow+\infty.
\]
Thus $\displaystyle\lim_{t\to+\infty}\tau(t) = +\infty$.
Conversely, since $r$ is
bounded below by a positive constant, then integrating $\dfrac{dt}{d\tau}=r^{3/2}$ gives
\[
t(\tau)=t_0+\int_{\tau_0}^{\tau}r(\sigma)^{3/2}\,d\sigma.
\]
Thus $t(\tau)\longrightarrow+\infty$ as $\tau\to+\infty$. Hence,  $t:[\tau_0,+\infty)\longrightarrow [t_0,+\infty)$
is a global orientation-preserving diffeomorphism on the interior intervals, with global inverse
\begin{equation}\label{eq:tau(t)}
        \tau=\tau(t),\qquad t\in[t_0,+\infty).
\end{equation}
Moreover,
\[
\frac{d\tau}{dt}=r(\tau(t))^{-3/2}=\rho(\tau(t))^3.
\]
Since the McGehee vector field is smooth in the collision-free region and the time-change satisfies
\[
\frac{dt}{d\tau}>0,
\]
the inverse \(\tau(t)\) is smooth on \((t_0,+\infty)\).
Thus $t\to+\infty$ and $\tau\to+\infty$ are equivalent in the situations
used below.
Consequently, $x(t)=r(\tau(t))s(\tau(t))$ is smooth on \((t_0,+\infty)\) wherever the
McGehee orbit is collision-free. In particular, it is legitimate to differentiate \(x(t)\) twice with respect to the physical time \(t\).

We now prove the converse correspondence directly, without using any formula that was itself
derived from the Newton equations
\begin{proposition}[Equivalence between McGehee and Newton equations]
\label{prop:Equivalence between McGehee and Newton}
Let
\[
Y(\tau)=(\rho(\tau),s(\tau),v(\tau),w(\tau))\in\mathcal P
\]
be a smooth solution of \eqref{eq:M1}--\eqref{eq:M4} on an interval on
which $\rho>0$.  Define $r=\rho^{-2}$,
$dt/d\tau=r^{3/2}$, and
\[
x(t)=r(\tau(t))s(\tau(t)).
\]
Then $x$ is a zero-energy solution of
\[
\ddot x=\nabla U(x).
\]
\end{proposition}

\begin{proof}
From \eqref{eq:M1} and \eqref{eq:M2},
\[
r'=rv,
\qquad
s'=w.
\]
Since $dt/d\tau=r^{3/2}$, this is equivalent to
\[
v=r^{1/2}\dot r,
\qquad
w=r^{3/2}\dot s.
\]
Differentiating the first identity with respect to $\tau$ gives
\[
v'=\frac12r\dot r^2+r^2\ddot r.
\]
Using \eqref{eq:M3} and the zero-energy constraint
$2U(s)=v^2+\|w\|^2$, we obtain
\[
r^2\ddot r = v' -\frac12 r\dot{r}^2
=\frac12\|w\|^2-\frac12v^2
=\|w\|^2-U(s)
=r^3\|\dot s\|^2-U(s).
\]
Equivalently,
\begin{equation}\label{A}
\ddot r-r\|\dot s\|^2=-\frac{U(s)}{r^2}.
\tag{A}
\end{equation}
Next,
\[
D_\tau w
=r^{3/2}D_t(r^{3/2}\dot s)
=\frac32r^2\dot r\dot s+r^3D_t\dot s.
\]
Substituting \eqref{eq:M4} and $vw=r^2\dot r\dot s$ yields
\[
r^3D_t\dot s+2r^2\dot r\dot s
=\nabla_{\mathcal E}U(s),
\]
or
\begin{equation}\label{B}
rD_t\dot s+2\dot r\dot s
=\frac1{r^2}\nabla_{\mathcal E}U(s).
\tag{B}
\end{equation}
The polar acceleration decomposition is
\begin{equation}\label{C}
\ddot x
=(\ddot r-r\|\dot s\|^2)s
+\bigl(rD_t\dot s+2\dot r\dot s\bigr).
\tag{C}
\end{equation}
Combining \eqref{A}--\eqref{C},
\begin{equation}\label{D}
\ddot x
=\frac1{r^2}\bigl(\nabla_{\mathcal E}U(s)-U(s)s\bigr).
\tag{D}
\end{equation}

 According to Euler's homogeneous function relation and the fact that $U$ is a homogeneous function of degree $-1$, we have
 the gradient decomposition
 \begin{equation}\label{E}
        \nabla U(rs)= r^{-2}\nabla U(s) =\frac1{r^2}\bigl(\nabla_{\mathcal E}U(s)-U(s)s\bigr).
\tag{E}
 \end{equation}
Therefore $\ddot x=\nabla U(x)$.  

Finally, since $\dot{x} = r\dot{s} + \dot{r}s = r^{-1/2}w + r^{-1/2} vs$, we have 
$\|\dot{x}\|^2 = r^{-1}(\|w\|^2 + v^2)$. Therefore
\[
\frac12\|\dot x\|^2-U(x)
=\frac1r\left[\frac12(v^2+\|w\|^2)-U(s)\right]=0,
\]
so the reconstructed solution has zero energy.
\end{proof}

\begin{proposition}\label{prop:ea+ is a equilibrium}
       For any central configuration $a$, $e_a^+ = (0,a,\sqrt{2U(a)},0)$ is a equilibrium of the McGehee flow 
       of the system \eqref{eq:M1}\eqref{eq:M2}\eqref{eq:M3}\eqref{eq:M4}. 
\end{proposition}
\begin{proof}
        At $e_a^+$, \eqref{eq:M1} gives
        $$\rho^{\prime}=-\frac12\cdot0\cdot v=0.$$
        \eqref{eq:M2} implies 
        $$s^{\prime}=w=0.$$
        Out of the system, \eqref{eq:M3} gives
        $$v^{\prime}=\frac12\|0\|^2=0.$$
        Finally, the equation \eqref{eq:M4} of $w$ and the fact of $a$ being a central configuration makes 
        $$D_\tau w=\nabla_\mathcal{E}U(a)-\frac12v\cdot0=0.$$

\end{proof}

\begin{proposition}\label{prop: parabolic motions hit ea+ at infty}
        Given any $c$-parabolic motion, its McGehee lift converges to $e_c^+$ at infinity.
        In particular, $v(\tau)\uparrow v_*$ as $\tau\to +\infty$.
\end{proposition}

\begin{proof}

Since for an $c$-parabolic motion, \(r(t)\to+\infty, s(t)\to c\), and $\displaystyle\lim_{\tau\to+\infty}t(\tau) = +\infty $ we have
\[
\rho(\tau) = r(t(\tau))^{-1/2}\rightarrow 0,~~s(\tau)\rightarrow c,~~U(s(\tau))\rightarrow U(c)~~(\tau\rightarrow+\infty).
\]

We next claim that this implies \(w(\tau)\to0\). 
The zero-energy identity \eqref{energy relation of McGehee coordinates} gives
that both \(v\) and \(w\) are bounded for large \(\tau\). By \eqref{eq:M3}, \(v\) is monotone nondecreasing and bounded above. Hence \(v(\tau)\) has a finite limit as \(\tau\to+\infty\). Consequently,
\[
\int_{\tau_0}^{+\infty}\|w(\tau)\|^2\,d\tau
=
2\int_{\tau_0}^{+\infty}v'(\tau)\,d\tau
=
2\left(\lim_{\tau\to+\infty}v(\tau)-v(\tau_0)\right)
<+\infty.
\tag{4}
\]
Thus
\[
w\in L^2([\tau_0,+\infty)).
\]
In a small collision-free neighborhood of \(c\), 
the McGehee vector field is smooth. Since \(s(\tau)\to c\), \(v(\tau)\) and \(w(\tau)\) 
remain bounded for large \(\tau\). The equation
\[
D_\tau w=\nabla_{\mathcal E}U(s)-\frac12vw
\]
therefore, \(D_\tau w\) is bounded for large \(\tau\). 
In local coordinates, this means that \(w'\) is bounded. Therefore \(w\) is 
uniformly continuous on \([\tau_0,+\infty)\) and hence $w(\tau)\to0$.

Taking the limit in the zero-energy identity \eqref{energy relation of McGehee coordinates}, and using \(s(\tau)\to c\) and \(w(\tau)\to0\), we obtain
\begin{equation*}
        \frac{1}{2}v(+\infty)^2 = U(c),
\end{equation*}
where $v(+\infty) = \displaystyle\lim_{\tau\to+\infty}v(\tau)$.
For an outgoing parabolic ray, \(r(t)\to+\infty\) in forward time, and the asymptotic McGehee radial 
velocity $v$ is positive. 
Otherwise 
\[
\rho' = -\frac{1}{2}\rho v
\]
makes $\rho\to+\infty$. Therefore
\[
v(+\infty)=v_*.
\]

Finally, since
\[
\rho=r^{-1/2}\to0,\qquad s\to c,\qquad w\to0,\qquad v\to v_*,
\]
the McGehee limit of the $c$-parabolic motions is $e_c^+$.
Moreover, along the McGehee flow 
\(v\) is nondecreasing since $ v'=\dfrac12\|w\|^2\geq 0$.
On the otherhand, since \(v(\tau)\to v_*\), we have $v(\tau)\leq v_*$,
as desired.

\end{proof}

As a contrary of \Cref{prop: parabolic motions hit ea+ at infty}, we here verify that
any trajectory of McGehee flow converging to $e_c^+$ at infinity can be natural 
interpreted to be a true $c$-parabolic motion.

\begin{proposition}[McGehee convergence to $e_c^+$ gives a parabolic motion]
\label{prop:McGehee orbit convergence to ea+ gives a parabolic motion}
Let $Y(\tau)\in\mathcal P$ be a smooth solution of
\eqref{eq:M1}--\eqref{eq:M4} for $\tau\geq\tau_0$, with $\rho(\tau)>0$,
and suppose
\[
Y(\tau)\longrightarrow e_c^+=(0,c,v_*,0).
\]
Then the reconstructed curve $x(t)=r(\tau(t))s(\tau(t))$ is a global
forward zero-energy Newtonian solution and
\[
\frac{x(t)}{t^{2/3}}\longrightarrow\alpha c.
\]
\end{proposition}

\begin{proof}
Since $\rho(\tau)\to0$, one has $r(\tau)\to+\infty$.  In particular,
$r$ is bounded below away from zero on a forward tail, so the time-change discussion
shows that $t(\tau)\to+\infty$ and gives the global inverse $\tau(t)$.
By \Cref{prop:Equivalence between McGehee and Newton}, the reconstructed
curve is a zero-energy Newtonian solution.

Moreover,
\[
\frac{d}{dt}r^{3/2}=\frac32r^{1/2}\dot r=\frac32v.
\]
Since $v(\tau(t))\to v_*$,
\[
\frac{r(t)^{3/2}}{t}\longrightarrow\frac32v_*,
\qquad
\frac{r(t)}{t^{2/3}}
\longrightarrow
\left(\frac32v_*\right)^{2/3}
=\left(\frac{9U_{min}}2\right)^{1/3}
=\alpha.
\]
Together with $s(\tau(t))\to c$, this gives
\[
\frac{x(t)}{t^{2/3}}
=\frac{r(t)}{t^{2/3}}s(\tau(t))
\longrightarrow\alpha c.
\]
\end{proof}

Finally, we establish an exponential decay estimate for orbits converging to $e_c^+$. 
Throughout the paper, we choose $\delta_0\in(0,v_*/2)$ so that 
$\overline{B_{\mathcal E}(c,\delta_0)}\subset\mathcal{E}_0$, where $B_{\mathcal E}(c,\delta_0)$ represents
the open ball in $\mathcal{E}$ centered at $c$.  We always assume that 
$\delta\in(0,\delta_0)$ unless otherwise stated, define
\begin{equation*}
        \mathcal{U}_c(\delta)=\{(\rho,s,v,w)\in\mathcal{P}:0\leq\rho<\delta,d_{\mathcal{E}}(s,c)<\delta,0\le v_*-v<\delta,\|w\|<\delta\}.
\end{equation*}
Every point in this set satisfies $v>v_*-\delta>0$ and hence belongs to
the outgoing branch $v=V(s,w)$ of the zero-energy surface.
\begin{proposition}\label{prop:outgoing zero-energy solutions locally exponentially converge}
Given any $\delta\in(0,\delta_0)$, 
Let \(Z:[0,\infty)\to\mathcal{P}\) be an orbit of the zero-energy McGehee system satisfying \(Z(\tau)\in\mathcal{U}_c(\delta)\) and \(Z(\tau)\to e_{c}^{+}\), with \(\rho=0\) allowed.
Then $Z(\tau)$ decay exponentially. More precisly, there is a constant $C$ only depending on $U$ and $\delta$,
such that for any $\tau>0$.
\begin{align*}
   &\rho(\tau) \leq C \mathrm{e}^{-\frac{1}{2}(v_*-\delta)\tau},~d_{\mathcal{E}}(s(\tau), c)\leq C\mathrm{e}^{-\frac{1}{6}(2v_*-\delta)\tau}\\ &0\leq v_* - v(\tau) \leq C\mathrm{e}^{-\frac{1}{2}(2v_*-\delta)\tau},~\|w(\tau)\|\leq C\mathrm{e}^{-\frac{1}{6}(2v_*-\delta)\tau}.\\ 
\end{align*}
\end{proposition}

\begin{proof}
Set $  F(\tau):=v_*-v(\tau)\geq0$, 
from \(v'\geq0\) and the limit \(v_*\), it follows that \(F=\nu_*-\nu\geq0\) and
 $F(\tau) \downarrow 0$ as $\tau\to +\infty$.
Using the minimality of \(c\) and the zero-energy relation and \(v_*^2=2U_{min}\), one obtains
\[
\begin{aligned}
        \|w\|^2
        &=2U(s)-v^2  \\
        &=2\bigl(U(s)-U_{min}\bigr)+(v_* + v)F\\
        &>(2v_* - \delta)F
\end{aligned}
\]
 therefore 
\begin{equation*}
        F'(\tau) = -\frac{1}{2}\|w(\tau)\|^2\leq-\frac{1}{2} (2v_* - \delta) F(\tau).
\end{equation*}
By Gronwall's inequality, we have that  for all $\tau>0$, 
\begin{align*}
       0\leq F(\tau)&\leq F(0)\mathrm{e}^{-\frac{1}{2}(2v_* - \delta)\tau} \\
        &=(v_* - v(0))\mathrm{e}^{-\frac{1}{2}(2v_* - \delta)\tau}\\
        &\leq\delta\mathrm{e}^{-\frac{1}{2}(2v_* - \delta)\tau}\\
        &:=C_1\mathrm{e}^{-\frac{1}{2}(2v_* - \delta)\tau}
\end{align*}

Next, we prove exponential decay of $w$. Since $F'(\tau) = -\dfrac{1}{2}\|w(\tau)\|^2$ and $F(\tau)\to 0$, the integrating gives
\begin{equation*}
        F(\tau) = \frac{1}{2}\int_\tau^\infty \|w(\sigma)\|^2d\sigma  ,~\tau>0,
\end{equation*}
hence 
\begin{equation}\label{eq:tail estimate of w}
        \int_\tau^\infty \|w(\sigma)\|^2d\sigma \leq 2C_1\mathrm{e}^{-\frac{1}{2}(2v_* - \delta)\tau}.
\end{equation}
Because the McGehee vector field is $C^k$ in $\mathcal{U}(\delta)$,  
\begin{equation*}
        \|D_\tau w\|\leq\sup_{d_\mathcal{E}(s,c)<\delta} \|\nabla_\mathcal{E}U(s)\| +\frac{1}{2}v_*\delta:=K ,~\tau>0,
\end{equation*}
$K$ only depend on $\delta$ and $U$. For any $\tau>0$, write $y(\tau) = \|w(\tau)\|$.
 $w(\tau)$ is $C^k$.
Parallel transporting \(w(\tau)\) along the orbit into a common tangent space yields, for \(\tau_2 \ge \tau_1\),
\[
\big|y(\tau_2) - y(\tau_1)\big| \le \int_{\tau_1}^{\tau_2} \|D_\tau w\| \, d\tau \le K|\tau_2 - \tau_1|.
\]
Consequently, given any $\tau\ge0$, on the interval $\left[ \tau, \tau + \dfrac{y(\tau)}{2K} \right]$, one obtains
\begin{equation*}
        y(\sigma) \geq y(\tau) - K(\sigma -\tau)\ge y(\tau) - K\frac{y(\tau)}{2K} = \frac{y(\tau)}{2}.
\end{equation*} 
Hence, 
\begin{equation*}
        \int_{\tau}^{\infty} \| w(\sigma) \|^2d\sigma \geq \int_{\tau}^{\tau +\frac{y(\tau)}{2K}} \frac{y^2(\tau)}{4} d\sigma = \frac{y^3(\tau)}{8K}.
\end{equation*}
The lower bound estimate obtained and the tail estimate gives
 \eqref{eq:tail estimate of w} yields 
\begin{equation*}
        \|w(\tau)\|=y(\tau) \leq C_2 \mathrm{e}^{-\frac{1}{6}(2v_* - \delta)\tau} ,~\tau>0,
\end{equation*}
where $C_2 = 2(2K\delta)^{\frac{1}{3}}$ is independent of the orbit.

Next, we deduce the exponential decay of $s$. Since $s' =w$ and $s(\tau)\rightarrow c$, 
\begin{align*}
        d_\mathcal{E}(s(\tau),c)&\leq \int_{\tau}^{\infty} \|s'(\sigma)\|d\sigma\\
        &= \int_{\tau}^{\infty} \|w(\sigma)\|d\sigma\\
        &\leq C_2\int_{\tau}^{\infty} \mathrm{e}^{-\frac{1}{6}(2v_* - \delta)\sigma}d\sigma\\
        &=C_3\mathrm{e}^{-\frac{1}{6}(2v_* - \delta)\tau},
\end{align*}
where $C_3 = \dfrac{12(2K\delta)^{\frac{1}{3}}}{2v_*-\delta}$ is independent of the orbit.

Finally, the radial equation \eqref{eq:M1} implies 
\begin{equation*}
        \rho'(\tau)\leq -\frac{1}{2}(v_* - \delta)\rho(\tau), ~\tau>0,
\end{equation*}
and using Gronwall's inequality again to have
\begin{equation*}
        \rho(\tau)\leq \rho(0) \mathrm{e}^{-\frac{1}{2}(v_* - \delta)\tau}\leq C_4 \mathrm{e}^{-\frac{1}{2}(v_* - \delta)\tau} ,~\tau>0,
\end{equation*} 
here $C_4 = \delta$.

Define $C:= \max\{ C_1, C_2, C_3, C_4\}$, then C only depend on $U$ and $\delta$, as desired.

\end{proof}

For $0<\delta<\delta_0$, define the local strong stable set
\begin{equation}\label{eq:strong stable set near ec+}
W_{loc}^{ss}(e_c^+;\delta)
:=
\left\{Z\in\mathcal P:
\Phi^\tau(Z)\in\mathcal U_c(\delta)\text{ for every }\tau \ge 0,
\quad
\Phi^\tau(Z)\to e_c^+
\right\}.
\end{equation}
The next two sections identify its tangent space and prove directly that it
is a smooth embedded strong stable manifold with boundary in the physical
half-space.

\section{Linearization at infinity}\label{secn:Linearization at infinity}

This section is devoted to the construction of the linearization at infinity. 
Let $G(\rho,s,v,w) = \dfrac{1}{2}(v^2+\|w\|^2) - U(s)$, the energy constraint is $G=0$. Since $v>v_*-\delta>0$, $v=V(s,w)$. 
Thus the zero energy hypersurface near $e_c^+$ is unique, yielding the physical outgoing branch. Thus 
$(\rho,s,w)$ is sufficient for our argument near $e_c^+$, we call it the reduced McGehee coordinates.

The existence of the manifold will be established in the phase space $\mathcal{P}$,
 after which we restrict the construction back to $\rho \geq 0$.
The extension of the McGehee flow from \(\rho>0\) to \(\rho\in\mathbb R\) is justified as follows.
 Equations \eqref{eq:M2}, \eqref{eq:M3}, \eqref{eq:M4} 
together with the zero-energy constraint are independent of $\rho$, and the equation
\[
\rho' = -\frac12 \rho v
\]
remains smooth for negative $\rho$.  More importantly, along any solution,
\[
\rho(\tau) = \rho(0) \exp\!\left(-\frac12 \int_0^\tau v(\sigma)\,d\sigma\right),
\]
hence the sign of $\rho$ is preserved; in particular, the sets
$\{\rho>0\}$, $\{\rho=0\}$ and $\{\rho<0\}$ are invariant under the flow.

The  reduced outgoing McGehee system are
\begin{align}
\rho' &= -\frac{1}{2}\rho V(s, w), 
\tag{M1'}\label{eq:M1'} \\
 s' &= w,  
\tag{M2'}\label{eq:M2'} \\
D_\tau w &= \nabla_\mathcal{E} U(s) - \frac{1}{2}V(s, w)w. 
\tag{M4'}\label{eq:M4'}
\end{align}
with the zero energy constraint.

\subsection{Linearization at $e_c^+$}\label{subscn:Linearization at ec+}
We write $(\rho, \xi, \delta v,  W)\in\mathbb{R}_\rho\oplus T_c\mathcal{E}_\xi\oplus \mathbb{R}_v\oplus T_c\mathcal{E}_W$ as the
first order variation or linearization of the coordinates $(\rho, s , v, w)$. Let $Z(\epsilon) = (\rho(\epsilon), s(\epsilon), v(\epsilon), w(\epsilon))$ 
be an arbitrary smooth curve satisfying the zero energy constraint such that $Z(0) = e_c^+$, write $ Z'(0)=(\rho, \xi, \delta v,  W) \in T_{e_c^+}\mathcal{P}$,
then 
\begin{equation*}
        \frac{d}{d\epsilon}\Big|_{\epsilon =0}G(Z(\epsilon)) = v_*\delta v - d(U|_{\mathcal{E}_0})_c[\xi] = 0. 
\end{equation*}
On the other hand, $v_*>0$, therefore $\delta v =0$ at $e_c^+$ and 
\begin{equation*}
        T_{e_c^+}\mathcal{P}=\left\{(\rho,\xi,\delta v,W):\delta v=0\right\}.
\end{equation*}
Throughout this paper, we denote $\mathscr{T}_0 =  \mathbb{R}_\rho \oplus T_c\mathcal{E}_\xi\oplus T_c\mathcal{E}_W$ for simplicity,
then there is a natural identification $T_{e_c^+}\mathcal{P} \cong \mathscr{T}_0$.

Given any sufficiently small $\xi\in T_c\mathcal{E}$, let $s=\exp_c\xi$, using parallel transport along 
the geodesic 
\begin{align*}
        \gamma_\xi:[0,1]&\rightarrow\mathcal{E}\\
        \gamma_\xi(\theta)&=\exp_c(\theta\xi),
\end{align*}
we move an arbitrary $w\in T_s\mathcal{E}$ to be $W\in T_c\mathcal{E}$, precisely, 
denote the parallel transport along $\gamma_\xi$ from $c$ to $s=\exp_c\xi$
\begin{equation*}
        P_\xi:T_c\mathcal{E}\rightarrow T_s\mathcal{E},
\end{equation*}
and its inverse $P_\xi^{-1}$. In particular,
$P_0 = \id_{T_c\mathcal{E}}$.
We define
\begin{equation*}
        W = P_\xi^{-1} w
\end{equation*} 
Hence there is a coordinates change near $e_c^+$ and it is locally diffeomorphism:
\begin{equation*}
        (\rho,\xi,W) \mapsto (\rho,s,w) = (\rho,\exp_c\xi,P_\xi W).
\end{equation*}
Note that the Riemannian metric is conserved along parallel transport,
 we define the radial velocity function $\hat{V}$ of $(\xi, W)$ near $e_c^+$ as follows, 
\begin{align*}
        \hat{V}(\xi,W) :&= V(\exp_c\xi,P_\xi W)\\
        &=\sqrt{2U(\exp_c\xi) - \|P_\xi W\|^2}\\
        &=\sqrt{2U(\exp_c\xi) - \|W\|^2}.
\end{align*}
Since $c$ is a MNCC, the Taylor formula for $U$  at $c$ is 
\begin{equation*}
        U(\exp_c\xi) = U_{min} + \frac{1}{2}\langle H\xi,\xi\rangle + O(\|\xi\|^3),
\end{equation*}
then 
\begin{equation}\label{estimate hat V}
\begin{aligned}
         v=\hat{V}(\xi,W) &= \sqrt{2 U_{min} + \langle H\xi, \xi \rangle + O(\|\xi\|^3) - \|W\|^2}\\
                       &= \sqrt{ v_*^2 + \langle H\xi, \xi \rangle + O(\|\xi\|^3) - \|W\|^2} \\
                       &= v_* +O(\|\xi\|^2 + \|W\|^2).
\end{aligned}
\end{equation}

First, \eqref{eq:M1'} and $V(s,w) = \hat{V}(\xi,W)$ gives
\begin{equation}\label{eq:Taylor formula of M1}
                \rho'  = -\frac{1}{2}\rho\left(v_* +O(\|\xi\|^2 + \|W\|^2)\right),
\end{equation}
yielding the linearized equation for $\rho$:
\begin{equation*}
        \rho' = -\frac{v_*}{2}\rho 
\end{equation*}

Second, differentiating  $s=\exp_c\xi$ with respect to $\tau$ yields
\begin{equation*}
        s' = (d\exp_c)_\xi[\xi']
\end{equation*}
Recall $s' = w = P_\xi W$, hence 
\begin{equation*}
        \xi' = [(d\exp_c)_{\xi}]^{-1} P_\xi [W], 
\end{equation*}
in particular,
\begin{equation*}
        (d\exp_c)_0 = I_{T_c\mathcal{E}}, ~P_0 = I_{T_c\mathcal{E}}.
\end{equation*}
Write $B(\xi) = [(d\exp_c)_{\xi}]^{-1} P_\xi:T_{c}\mathcal{E}\to T_{c}\mathcal{E}$, 
then $\xi' = B(\xi) W$.
Note that
$(d\exp_c)_0 = P_0 = \id_{T_c\mathcal{E}}$, then
$B(0) = \id_{T_c\mathcal{E}}$ and $B(\xi) = \id_{T_c\mathcal{E}} + O(\|\xi\|)$,
yielding the variation of equation for shape $s$
\begin{equation}\label{eq:Taylor formula of M2}
        \xi' = W + O(\|\xi\|\|W\|).
\end{equation}
Hence the linearized equation for $\xi$ is 
\begin{equation*}
        \xi' =W. 
\end{equation*}

Third, taking the derivative w.r.t. $\tau$ on both side of $W = P_\xi^{-1} w$ gives 
\begin{equation}\label{eqn 1 for compute W'}
        W' = (\nabla P^{-1})_\xi[\xi'] w + P_\xi^{-1} D_\tau w.
\end{equation}
Since the transport preserves the inner product, we have
\[
\|P_\xi\| = \|P^{-1}_\xi\| = 1, \qquad \|w\| = \|P_\xi W\| = \|W\|.
\]
Moreover, since $\xi' = W + O(\|\xi\|\|W\|)$, 
it follows that the first term of \eqref{eqn 1 for compute W'} is 
\begin{equation*}
\left|(\nabla P^{-1})_\xi[\xi'] w\right|
=O(\|\xi'\| \|w\|) 
= O(\|W\|^2).
\end{equation*}
Recall \eqref{eq:M4'} to have
\begin{equation}\label{eqn 3 for compute W'}
        \begin{aligned}
        P_\xi^{-1} D_\tau w &= P_\xi^{-1} \nabla_\mathcal{E}U(s) - \frac{1}{2}V(s,w) P_\xi^{-1} w\\
            & =  P_\xi^{-1} \nabla_\mathcal{E}U(\exp_c\xi) - \frac{1}{2}\hat{V}(\xi,W) W,
\end{aligned}
\end{equation}
where the first term is a pulled back gradient and satisfies 
\begin{equation}\label{eqn 4 for compute W'}
        P_\xi^{-1} \nabla_\mathcal{E} U(\exp_c\xi) = H\xi + O(\|\xi\|^2).
\end{equation}
Using \eqref{estimate hat V} to have 
\begin{equation}\label{eqn 5 for compute W'}
        -\frac{1}{2}\hat{V} (\xi,W) W = -\frac{1}{2}v_*W + O(\|\xi\|^2\|W\| + \|W\|^3).
\end{equation}
Combining \eqref{eqn 1 for compute W'} to \eqref{eqn 5 for compute W'} gives 
\begin{equation}\label{eq:Taylor formula of M4}
        W'  = H\xi - \frac{v_*}{2}W + O(\|\xi\|^2 + \|W\|^2),
\end{equation}
there the lineared version of equation \eqref{eq:M4'} is derived to be 
\begin{equation*}
        W' = H\xi -\frac{v_*}{2}W.
\end{equation*}

Summerize the derived linearizations of system \eqref{eq:M1'}\eqref{eq:M2'}\eqref{eq:M4'} as the following
\begin{empheq}[left=\empheqlbrace]{align}
        \rho' &= -\frac{v_*}{2}\rho, \tag{LM1}\label{eq:LM1}\\
         \xi'       &= W   \tag{LM2}\label{eq:LM2},\\
         W'         &= H\xi -\frac{v_*}{2}W  \tag{LM4}\label{eq:LM4},
\end{empheq}
We denote the matrix of the coefficients of the linear part by $A$.

\subsection{Computing the stable linear space at $e_c^+$}

We carefully distinguish the eigenvalues of the Hessian $H$ from the
spectral values of the linear system
\eqref{eq:LM1}\eqref{eq:LM2}\eqref{eq:LM4}. Choose an orthonormal basis
$\{e_j\}_{j=1}^{m-1}$ of $T_c\mathcal E$ consisting of eigenvectors of $H$:
\begin{equation*}
        He_j = \mu_j e_j, ~\mu_j\geq0.
\end{equation*}
Write $Q_H=(e_1,e_2,\dots,e_{m-1})$, since 
$H$ has its orthogonal decomposition $H=Q_H MQ_H^{-1}$ with $M = \diag\{\mu_1,\mu_2,\dots,\mu_{m-1}\}$,
its characteristic polynomial is 
\begin{align*}
       \det (\lambda I_{2m-1}- A) &= (\lambda + \frac{v_*}{2})\det\big[(\lambda^2 + \frac{v_*}{2}\lambda)I_{m-1} - H\big]\\
&=(\lambda + \frac{v_*}{2})\det\big[(\lambda^2 + \frac{v_*}{2}\lambda)Q_HQ_H^{-1} - Q_HMQ_H^{-1}\big]\\
&=(\lambda + \frac{v_*}{2})\Pi_{j=1}^{m-1}(\lambda^2 + \frac{v_*}{2}\lambda - \mu_j).
\end{align*}
Therefore $\mathrm{Spec}A\subset\mathbb{R}$.
In the direction $e_j$, write $\xi=\xi_j e_j$ and $W=W_j e_j$. Equations
\eqref{eq:LM2} and \eqref{eq:LM4} become
\[
\xi_j'=W_j,\qquad
W_j'=\mu_j\xi_j-\frac{v_*}{2}W_j,
\]
and hence 
\begin{equation*}
        \xi_j''+\frac{v_*}{2}\xi_j' - \mu_j\xi_j =0.
\end{equation*}
Let 
\begin{align*}
        \xi_j(\tau)&=Ce^{\lambda\tau}\\
        W_j(\tau) &= \xi_j'(\tau) = \lambda\xi_j(\tau)
\end{align*}
to get the characteristic equation
\begin{equation*}
        \lambda^2 + \frac{v_*}{2} \lambda - \mu_j = 0,
\end{equation*}
whose roots are 
\begin{equation*}
        \lambda_j^{\pm}(\mu_j) = \frac{-\dfrac{v_*}{2} \pm \sqrt{\dfrac{v_*^2}{4} + 4\mu_j}}{2},
\end{equation*}
then 
\begin{align*}
        \lambda_j^-<-\frac{v_*}{2}<0<\lambda_j^+~~&\text{for}~~\mu_j>0,\\
        \lambda_j^- = -\frac{v_*}{2}, \lambda_j^+ = 0~~&\text{for}~~\mu_j=0.
\end{align*}
Thus every shape direction has one strictly negative root, whereas the
second root is positive precisely on $(\ker H)^\perp$ and is zero on
$\ker H$.
Finally,
\[
\operatorname{Spec}(A)
\subset(-\infty,-v_*/2]\cup\{0\}\cup(0,+\infty).
\]

Define the self-adjoint operator
\[
\Lambda^-:T_c\mathcal E\longrightarrow T_c\mathcal E,
\qquad
\Lambda^-e_j=\lambda_j^-e_j,
\]
hence 
\begin{equation*}
        W = \Lambda^-\xi.
\end{equation*}
On the otherhand, since $\rho' = -\dfrac{v_*}{2}\rho$, $\rho$ admits the 
stable eigenvalues $-\dfrac{v_*}{2}<0$. Therefore, the negative spectral subspace of the full linearized system is
\begin{equation*}
        E_c^{ss}  = \Big\{ (\rho,\xi,W):\rho\in\mathbb{R},\xi\in T_c\mathcal{E}, W=\Lambda^-\xi \Big\}.
\end{equation*}
Therefore
\[
\dim E_c^{ss}=1+\dim T_c\mathcal E=1+(m-1)=m.
\]

\subsection{Linear space decomposition/spectral splitting}

With the linearization established above, we now describe the invariant
spectral decomposition of the tangent space at $e_c^+$.

\begin{proposition}[Linear space decomposition/spectral splitting]
\label{prop:Linear space decomposition/spectral splitting}
Let $c$ be a MNCC and put
\[
K:=\ker H,
\qquad
P:=K^\perp\subset T_c\mathcal E.
\]
Then
\begin{equation}
\label{decompostion of the tangent space at ec+}
T_{e_c^+}\mathcal P
=E_c^{ss}\oplus E_c^c\oplus E_c^u.
\end{equation}
Here
\begin{align*}
E_c^{ss}
&=\bigl\{(\rho,\xi,\Lambda^-\xi):
\rho\in\mathbb R,\ \xi\in T_c\mathcal E\bigr\},\\
E_c^c
&=\bigl\{(0,\xi_0,0):\xi_0\in K\bigr\},\\
E_c^u
&=\bigl\{(0,\xi_+,\Lambda^+\xi_+):\xi_+\in P\bigr\},
\end{align*}
where $\Lambda^-:T_c\mathcal E\to T_c\mathcal E,~~  \Lambda^- e_j = \lambda_j^- e_j$ is defined by the
negative roots $\lambda_j^-$ and
$\Lambda^+$ by the positive ones. The three subspaces are $A$-invariant,
$\dim E_c^{ss}=m$, and the decomposition is uniquely determined by $c$.
We also write
\[
E_c^{cu}:=E_c^c\oplus E_c^u.
\]
\end{proposition}

\begin{proof}
Since $H$ is self-adjoint and nonnegative,
$T_c\mathcal E=K\oplus P$, where $P$ is the direct sum of the eigenspaces
corresponding to the positive eigenvalues of $H$.

Solving $A(\rho,\xi,W)=0$ gives
\[
-\frac{v_*}{2}\rho=0,
\qquad W=0,
\qquad H\xi-\frac{v_*}{2}W=0,
\]
so
\[
\ker A=\{(0,\xi_0,0):\xi_0\in K\}.
\]
There is no nontrivial generalized eigenvector at the eigenvalue zero.
Indeed, if $Ay=(0,\eta,0)\in\ker A$, then
$W=\eta$ and $H\xi-(v_*/2)\eta=0$. Taking the inner product with
$\eta\in K$ gives 
\[
\langle H\xi,\eta \rangle - \frac{v_*}{2}\|\eta\|^2 = 
\langle \xi,H\eta \rangle - \frac{v_*}{2}\|\eta\|^2 =
- \frac{v_*}{2}\|\eta\|^2,
\]
yielding
\[
-\frac{v_*}{2}\|\eta\|^2=0,
\]
hence $\eta=0$ and the center space $E_c^c = \ker A$. 
On each positive eigenspace of $H$, the two roots satisfy
$\lambda_j^-<0<\lambda_j^+$. 
Hence we have precisely the unstable space 
\[
E_c^u=\{(0,\xi_+,\Lambda^+\xi_+):\xi_+\in P\}.
\]
Since $\rho' = -\dfrac{v_*}{2}\rho$ and $v_*>0$, the radial direction has eigenvalue
$-v_*/2<0$, it refuses center or unstable directions, hence $\rho$
belongs to $E_c^{ss}$ and $E_c^{ss}
=\bigl\{(\rho,\xi,\Lambda^-\xi):
\rho\in\mathbb R,\ \xi\in T_c\mathcal E\bigr\}$. 
It remains to verify the direct sum. 

A natural decomposition of $T_{e_c^+}\mathcal{P}$ is
\begin{align*}
        T_{e_c^+}\mathcal{P} &= \mathbb{R}_\rho\oplus T_c\mathcal{E}_\xi\oplus T_c\mathcal{E}_W\\
&=\mathbb{R}_\rho\oplus (K_\xi\oplus K_W)\oplus (P_\xi\oplus P_W)
\end{align*}
On $K_\xi\oplus K_W$, since $\Lambda^+|_K=0$
and $\Lambda^-|_K=-\frac{v_*}{2}I_K$, every pair
$(\xi_0,W_0)$ admits the unique decomposition
\[
(\xi_0,W_0)
=(\alpha,\Lambda^-\alpha)+(\beta,0),
\]
where
\[
\alpha=-\frac{2}{v_*}W_0,
\qquad
\beta=\xi_0-\alpha.
\]
The first summand belongs to the stable space,
and the second to the center space.
On $P_\xi\oplus P_W$, the operator
$\Lambda^+-\Lambda^-$ is invertible,
so every pair decomposes uniquely into
stable and unstable components. Combining the radial, kernel, and positive
eigenvalue parts proves the direct-sum decomposition and its uniqueness.
 All three summands are invariant under $A$.

\end{proof}

\section{Strong stable manifold and its tangent space at $e_c^+$}\label{secn:strong stable manifold and its tangent space at ec+}
In this section, we prove our theorem of strong stable manifold. 
The Newtonian potential is smooth on \(\mathcal{E}_0\), in this section, we let \(k \ge 3\) be an arbitrary finite integer. 
Due to the arbitrariness of \(k\), we can directly treat the problem in the \(C^\infty\) category in the next section.
\begin{theorem}\label{thm:strong stable set is a submfd and its tangent space is Ecss}
      Let $0<\delta<\delta_0$,  let $W_{loc}^{ss}(e_c^+;\delta)$ be as in \eqref{eq:strong stable set near ec+}, then $W_{loc}^{ss}(e_c^+;\delta)$ is a $m$ dimensional $C^k$ embedded
         submanifold  with boundary of $\mathcal{P}$ and $T_{e_c^+}W_{loc}^{ss}(e_c^+;\delta) = E_c^{ss}$. 
\end{theorem}

Recall the linearization near $e_c^+$, that for any $s$ near $c$, $s = \exp_c\xi$, $\xi\in T_c\mathcal{E}$,
$P_\xi:T_c\mathcal{E} \to T_s\mathcal{E}$ the parallel transport along $\theta\mapsto \exp_c(\theta \xi)$,
$w=P_\xi W$, $W \in T_c\mathcal{E}$, $\|w\| =\|W\|$, and remember $v = V(s,w)$.
Then we have the local coordinates chart in a open neighborhood of $e_c^+$ in $\mathcal{P}$: 
\begin{align*}
            \chi:\mathcal{O}&\longrightarrow  \mathscr{T}_0\quad\\
            (\rho,s,v,w)&\longmapsto Y=(\rho,\xi,W).
\end{align*}
In particular, $\chi$ is $C^k$ and $\chi(e_c^+) = 0$. 
To facilitate the use of the new chart, we assum that all configurations corresponding to $\mathcal{O}$ belong to $\mathcal{E}_0$.
We further assume throughout that $\delta_0\in(0,v_*/2)$ is small enough
so that $\mathcal{U}_c(\delta_0)$ lies within the coordinate chart $\chi$.

 Let $f:\chi(\mathcal{O})\subset  \mathscr{T}_0 \rightarrow  \mathscr{T}_0 $ be our McGehee vector field expressed in the
  coordinates
  $\chi$:
 \begin{equation}\label{eq:McGehee system under coordinates chi}
    Y' = f(Y).
 \end{equation}
Here
$f$ is the pushforward of the original vector field $\mathfrak{X}$ by the coordinate map $\chi$:
$f = \chi_* \mathfrak{X}$;
more precisely,
\begin{equation}\label{eq:f is the pushforward}
    f(Y) = d\chi|_{\chi^{-1}(Y)} \, \mathfrak{X}\bigl(\chi^{-1}(Y)\bigr).
\end{equation}
Denote $N(Y) :=f(Y) -AY$ and the open ball $B_R\subset  \mathscr{T}_0 $ centered at $0\in\mathscr{T}_0$ with radius $R$.
Unless stated otherwise in our paper, we always assume that $R$ is such that the ball $\overline{B_R}$ is contained in the $\chi(\mathcal{O})$.

Before proceeding to the next subsection, we list some estimates that will be needed.
Remember \Cref{prop:Linear space decomposition/spectral splitting} gives us $ \mathscr{T}_0  = E_c^{ss}\oplus E_c^{cu}$.
Denote by $P_{ss}$ and $P_{cu}$  the  linear projections associated with $E_c^{ss}$ and $E_c^{cu}$. Let $\eta_*:=\dfrac{v_*}{2}$,
all negative eigenvalues are bounded above by $-\eta_*$.
The subspaces $E_c^{ss}$ and $E_c^{cu}$ are the direct sums of all eigenspaces
corresponding to the negative and non-negative eigenvalues of $A$, respectively. Hence $E_c^{ss}$ and $E_c^{cu}$ are 
invariant under $A$. 
Consequently, the operator $A$ and the projections satisfy some useful properties.

First, there exists a constant $M \geq 0$ such that
\begin{equation*}
    \|\mathrm{e}^{A\tau} P_{ss}\| \leq M \mathrm{e}^{-\eta_* \tau}, \qquad
    \|\mathrm{e}^{-A\tau} P_{cu}\| \leq M, \qquad \tau \geq 0,
\end{equation*}
where
\begin{equation*}
    \|\mathrm{e}^{A\tau} P_{ss}\| = \sup_{\|Z\|=1} \|\mathrm{e}^{A\tau} P_{ss} Z\|, \qquad
    \|\mathrm{e}^{-A\tau} P_{cu}\| = \sup_{\|Z\|=1} \|\mathrm{e}^{-A\tau} P_{cu} Z\|.
\end{equation*}
From the exponential series
\[
\mathrm{e}^{A\tau} = I + \tau A + \frac{\tau^2 A^2}{2!} + \cdots,
\]
we have
\[
\mathrm{e}^{A\tau} E_c^{ss} \subset E_c^{ss}\quad\text{and}\quad \mathrm{e}^{A\tau} E_c^{cu} \subset E_c^{cu}
\]
for all $\tau \in \mathbb{R}$, and hence
\begin{equation*}
    \mathrm{e}^{A\tau} E_c^{ss} = E_c^{ss}\quad\text{and}\quad \mathrm{e}^{A\tau} E_c^{cu} = E_c^{cu}.
\end{equation*}
Finally, since $\mathscr{T}_0 = E_c^{ss} \oplus E_c^{cu}$ and both $E_c^{ss}$ and $E_c^{cu}$ are $A$-invariant,
both $A$ and $\mathrm{e}^{A\tau}$ commute with the projection $P_{ss}$  and $P_{cu}$.

For $1\leq n \leq k+1$, we  wirte two notations:
\begin{equation*}
C_n^R = \sup_{\|Z\|\leq R}\|D^nN(Z)\|<+\infty.
\end{equation*}
The following statements hold.
 \begin{lemma}\label{lem:f N are smooth}
\begin{enumerate}[label=\textnormal{(\roman*)}]
    \item $f$ and $N$ are smooth on $\chi(\mathcal{O})$;
    \item $f(0) = 0$, $A = Df(0)$, $N(0) = 0$, $DN(0) = 0$;
    \item For any $Y_1,Y_2\in  \overline{B}_R $,
    \begin{equation*}
        \|N(Y_1) - N(Y_2)\|_{ \mathscr{T}_0 }\leq C_1^R\|Y_1-Y_2\|_{ \mathscr{T}_0 }.
    \end{equation*}
\end{enumerate}
\end{lemma}

\begin{proof}
Smoothness follows from the smoothness of the Newtonian vector field on the
collision-free chart and from the smooth dependence of $\exp_c$ and radial
parallel transport on $\xi$.  Since $e_c^+$ is an equilibrium and $A=Df(0)$,
we have
\[
f(0)=0,\qquad N(0)=0,\qquad DN(0)=Df(0)-A=0.
\]
Finally, the segment joining two points of $\overline{B_R}$ remains in that
ball.  Hence 
\[
N(Y_1)-N(Y_2)
=\int_0^1 DN\bigl(Y_2+t(Y_1-Y_2)\bigr)[Y_1-Y_2] \,dt,
\]
and therefore
\[
\|N(Y_1)-N(Y_2)\|_{ \mathscr{T}_0 }
\leq C_1^R\|Y_1-Y_2\|_{ \mathscr{T}_0 }.
\]
\end{proof}
\subsection{The Lyapunov-Perron operator}

We make full use of the Lyapunov-Perron operator, which links the strong stable manifold with the solutions of 
McGehee system. Write the operator as LP for short. 
By solving the equation via the Lyapunov--Perron operator, we obtain a $C^k$ submanifold in $\mathscr{T}_0$,
which will serve as the counterpart of the submanifold of $\mathcal{P}$ constructed in the next subsection.

Let $\gamma_\delta:=\min\left\{\dfrac{v_* - \delta}{2}, \dfrac{2v_*-\delta}{6}\right\}$,
then $v_*/4<\gamma_\delta<\eta_*$ and $\gamma_\delta$ is no more than any of the exponential decay rates that appear 
in \Cref{prop:outgoing zero-energy solutions locally exponentially converge}. 
Fix 
\[
\beta\in(v_*/4,\gamma_\delta).
\]
Define
\begin{equation*}
    \mathscr{X}_{\beta} = \left\{ Y\in C^0\big([0,+\infty), \mathscr{T}_0\big):\|Y\|_\beta: = \sup_{\tau\geq0}\mathrm{e}^{\beta\tau} \|Y(\tau)\|_{\mathscr{T}_0} < +\infty \right\}.
\end{equation*} 
\begin{lemma}\label{lem:X_beta is Banach}
    The normed space $(\mathscr{X}_\beta, \|\cdot\|_\beta)$ is a Banach space.
\end{lemma}
\begin{proof}
    For any Cauchy sequence $\{Y_n\}$ in $\mathscr{X}_\beta$, $\mathrm{e}^{\beta\tau} \|Y_n(\tau) - Y_m(\tau)\|_{\mathscr{T}_0}\leq \|Y_n-Y_m\|_\beta$,
hence for $\tau\geq0$, $\{Y_n(\tau)\}$ is a Cauchy sequence in the finite dimensional $\mathscr{T}_0$, then there is a unique $Y(\tau)\in \mathscr{T}_0$ such that 
$Y_n(\tau)\to Y(\tau) $. On the other hand, $\{Y_n\}$ is bounded in $\mathscr{X}_\beta$, hence there is some constant $C$, 
 $\mathrm{e}^{\beta\tau}\|Y_n(\tau)\|_{\mathscr{T}_0} \leq C$ holds for any $\tau\geq0$ and so is $\mathrm{e}^{\beta\tau}\|Y(\tau)\|_{\mathscr{T}_0}$.
 Let $m\to\infty$ in $\|Y_n - Y_m\|_\beta$ to have that $Y_n\to Y$ in the norm $\|\cdot\|_\beta$.
 Notice that $\|Y_n(\tau) - Y(\tau)\|_{\mathscr{T}_0} = \mathrm{e}^{-\beta\tau} \mathrm{e}^{\beta\tau} \|Y_n(\tau) - Y(\tau)\|_{\mathscr{T}_0} \leq \|Y_n - Y\|_\beta$,
 naming that $Y_n\to Y$ uniformly in $[0,+\infty)$, then $Y\in C^0([0,+\infty),\mathscr{T}_0)$. Therefore $Y\in\mathscr{X}_\beta$ and $\mathscr{X}_\beta$ is complet.
\end{proof}

For $R>0$, let
\(
B_R^\beta:=\{Y\in\mathscr X_\beta:\|Y\|_\beta<R\}
\).
The closed ball $\overline B_R^\beta$ is complete.  We choose $R$ so
small that $\overline B_R\subset\chi(\mathcal O)$.  Then, for every
$Y\in\overline B_R^\beta$ and every $\tau\geq0$,
\[
\|Y(\tau)\|\leq R e^{-\beta\tau}\leq R,
\]
so that $Y(\tau)\in\overline B_R$.

Given any $a\in E_c^{ss}$, the Lyapunov--Perron operator used in this paper is defined as follows:
 \begin{equation*}
   \begin{aligned}
     \mathcal{T}_a: B_R^\beta &\rightarrow \mathscr{X}_\beta\\
     [\mathcal{T}_a(Y)](\tau) = \mathrm{e}^{A\tau}a  &+ \int_0^\tau \mathrm{e}^{A(\tau-\sigma)} P_{ss} N(Y(\sigma)) d\sigma\\
                                                   &-\int_{\tau}^{+\infty} \mathrm{e}^{A(\tau-\sigma)} P_{cu} N(Y(\sigma))d\sigma.
   \end{aligned}
 \end{equation*}
The integrals converges  because $ \| \mathrm{e}^{-A(\sigma - \tau)} P_{cu}\|\leq M$ and
by \Cref{lem:f N are smooth}, $N(0)=0$ and 
\(
\|N(Y(\sigma))\|
\leq C_1^R\|Y(\sigma)\|
\leq C_1^R R e^{-\beta\sigma}.
\)
Write 
\begin{equation*}
     q(n,R,\beta) = MC_n^R\left(\frac{1}{\eta_* - \beta} + \frac{1}{n\beta}\right)>0.
\end{equation*}
In particular, since $DN(0) = 0$,  
$$\lim_{R\to 0}q(1,R,\beta) =0.$$
 The following statement holds.
\begin{lemma}\label{lem:estimate of D^n_YT_a(Y)}
    For $Y\in B_R^\beta$ and $1\leq n\leq k+1$, the derivative
$D_Y^n\mathcal T_a(Y)$ is the bounded $n$-linear operator given by, for any $u_1,\dots,u_n\in\mathscr{X}_\beta$,

\begin{equation*}
\begin{aligned}
        D^n_Y\mathcal{T}_a(Y)[u_1,\dots,u_n](\tau) &=\int_0^\tau \mathrm{e}^{A(\tau - \sigma)} P_{ss} D^nN(Y(\sigma)) [u_1(\sigma),\dots,u_n(\sigma)]d\sigma\\
        &-\int_{\tau}^{+\infty} \mathrm{e}^{A(\tau - \sigma)} P_{cu} D^nN(Y(\sigma)) [u_1(\sigma),\dots,u_n(\sigma)]d\sigma,
        \end{aligned}
\end{equation*}
 and
\begin{equation*}
    \|D^n_Y\mathcal{T}_a(Y)\|\leq q(n,R,\beta).
\end{equation*}
Consequently, $(a,Y)\mapsto\mathcal T_a(Y)$ is $C^k$ on
$E_c^{ss}\times B_R^\beta$.
\end{lemma}
\begin{proof}
    Differentiation under the two integrals is justified by dominated convergence.
Every $Y\in B_R^\beta$ takes values in $B_R$ and the derivatives of
$N$ are bounded there.  Remember that $\|\mathrm{e}^{A\tau} P_{ss}\|\leq M\mathrm{e}^{-\eta_*\tau}$ and $\|\mathrm{e}^{-A\tau} P_{cu}\|\leq M$.
For $u_j\in\mathscr X_\beta$, 
\[
\|u_j(\sigma)\|\leq e^{-\beta\sigma}\|u_j\|_\beta.
\]
Consequently, the stable integral is bounded by
\begin{equation*}
  \begin{aligned}
         &\mathrm{e}^{\beta\tau} \left\| \int_0^\tau \mathrm{e}^{A(\tau-\sigma)} P_{ss} D^nN(Y(\sigma)) [u_1(\sigma),\dots,u_n(\sigma)]  d\sigma \right\|_{\mathscr{T}_0}\\
         \leq & \int_{0}^{\tau} \left\|\mathrm{e}^{A(\tau - \sigma)}  P_{ss} e^{\beta\tau} D^nN(Y(\sigma)) [u_1(\sigma),\dots,u_n(\sigma)] \right\|_{\mathscr{T}_0} d\sigma\\
       \leq & MC_n^R\Pi_{i=1}^n \|u_i\|_\beta \int_0^\tau \mathrm{e}^{\beta\tau} \mathrm{e}^{-\eta_*(\tau - \sigma)} \mathrm{e}^{-n\beta\sigma}  d\sigma\\
      = &MC_n^R\Pi_{i=1}^n \|u_i\|_\beta \int_0^\tau  \mathrm{e}^{-(\eta_*-\beta)(\tau - \sigma)} \mathrm{e}^{-(n-1)\beta\sigma}  d\sigma\\
    \leq&\frac{MC_n^R}{\eta_* - \beta}\Pi_{i=1}^n \|u_i\|_\beta. 
  \end{aligned}
\end{equation*}
For the center-unstable integral,
\begin{equation*}
\begin{aligned}
         &\mathrm{e}^{\beta\tau} \left\| \int_{\tau}^{+\infty} \mathrm{e}^{A(\tau-\sigma)} P_{cu} D^nN(Y(\sigma)) [u_1(\sigma),\dots,u_n(\sigma)]  d\sigma \right\|_{\mathscr{T}_0}\\
\leq & MC_n^R \Pi_{i=1}^n \|u_i\|_\beta \mathrm{e}^{\beta\tau} \int_{\tau}^{+\infty}  \mathrm{e}^{-n\beta\sigma} d\sigma\\
\leq &\frac{MC_n^R}{n\beta} \mathrm{e}^{-(n-1)\beta\tau} \Pi_{i=1}^n \|u_i\|_\beta\\
\le & \frac{MC_n^R}{n\beta}  \Pi_{i=1}^n \|u_i\|_\beta.
\end{aligned}
\end{equation*}

Taking the supremum in $\tau$ proves the claimed bound.  The same dominated
convergence estimates, applied locally in $(a,Y)$, give the asserted
$C^k$ dependence.

\end{proof}

Put
\(\displaystyle
\varrho(\beta,R)
:=\frac{R}{M}\bigl(1-q(1,R,\beta)\bigr).
\)
Choose $R>0$ so small that $q(1,R,\beta)<1$ and 
$B_{\varrho(\beta,R)}\subset E_c^{ss}$ is contained in
$\chi(\mathcal O)$.
\begin{proposition}\label{prop:T_a is a contraction}
For every $a\in B_{\varrho(\beta,R)}$, the map
\[
\mathcal T_a:\overline B_R^\beta\longrightarrow\overline B_R^\beta
\]
is a contraction with contraction constant $q(1,R,\beta)$.  In fact,
$\mathcal T_a(\overline B_R^\beta)\subset B_R^\beta$.
\end{proposition}
\begin{proof}

For $Y\in\overline B_R^\beta$, the stable integral satisfies
\begin{equation*}
\begin{aligned}
            &\mathrm{e}^{\beta\tau} \left\| \int_0^\tau\mathrm{e}^{A(\tau-\sigma)} P_{ss} N(Y(\sigma))d\sigma  \right\|_{\mathscr{T}_0}\\
       \leq &\mathrm{e}^{\beta\tau}\int_0^\tau M\mathrm{e}^{-\eta_*(\tau-\sigma)} \|N(Y(\sigma))\|_{\mathscr{T}_0}d\sigma\\
       \leq &\mathrm{e}^{\beta\tau}\int_0^\tau M\mathrm{e}^{-\eta_*(\tau-\sigma)} C_1^R \|Y(\sigma)\|_{\mathscr{T}_0}d\sigma\\
        =   &MC_1^R \int_0^\tau \mathrm{e}^{-\eta_*(\tau-\sigma)} \mathrm{e}^{\beta(\tau - \sigma)} \mathrm{e}^{\beta\sigma} \| Y(\sigma) \|_{\mathscr{T}_0} d\sigma\\
        \le & MC_1^R \|Y\|_\beta\int_0^\tau \mathrm{e}^{-(\eta_*-\beta)(\tau-\sigma)}  d\sigma\\
        \leq&\frac{MC_1^R}{\eta_* - \beta}\|Y\|_\beta.
\end{aligned} 
\end{equation*}
Similarly, the center–unstable integral satisfies
\begin{equation*}
    \begin{aligned}
        &\mathrm{e}^{\beta\tau} \left\| \int_{\tau}^{+\infty} \mathrm{e}^{A(\tau-\sigma)} P_{cu} N(Y(\sigma)) d\sigma \right\|_{\mathscr{T}_0}\\
   \leq &\mathrm{e}^{\beta\tau} \int_{\tau}^{+\infty} M \left\|N(Y(\sigma))\right\|_{\mathscr{T}_0}d\sigma\\
   \leq &M\mathrm{e}^{\beta\tau}\int_{\tau}^{+\infty} C_1^R \left\|Y(\sigma)\right\|_{\mathscr{T}_0}d\sigma\\
   \leq &M C_1^R\mathrm{e}^{\beta\tau}\int_{\tau}^{+\infty} \mathrm{e}^{-\beta\sigma} \mathrm{e}^{\beta\sigma}\left\|Y(\sigma)\right\|_{\mathscr{T}_0}d\sigma\\
   \leq &M C_1^R\|Y\|_\beta \mathrm{e}^{\beta\tau} \int_{\tau}^{+\infty} \mathrm{e}^{-\beta\sigma} d\sigma\\
    =   &\frac{MC_1^R}{\beta}\|Y\|_\beta.
\end{aligned}
\end{equation*}
Therefore 

\begin{equation*}
    \begin{aligned}
        &\mathrm{e}^{\beta\tau} \left\| [\mathcal{T}_a(Y)](\tau)\right\|_{\mathscr{T}_0} \\
    \leq &\|\mathrm{e}^{\beta\tau}\mathrm{e}^{A\tau}a\|  + 
        \mathrm{e}^{\beta\tau} \left\| \int_0^\tau\mathrm{e}^{A(\tau-\sigma)} P_{ss} N(Y(\sigma))d\sigma  \right\|_{\mathscr{T}_0}\\
        \qquad &+\mathrm{e}^{\beta\tau} \left\| \int_{\tau}^{+\infty} \mathrm{e}^{A(\tau-\sigma)} P_{cu} N(Y(\sigma)) d\sigma \right\|_{\mathscr{T}_0}\\
    \leq &M\mathrm{e}^{\beta\tau} \mathrm{e}^{-\eta_*\tau}\|a\|_{\mathscr{T}_0} + q(1,R,\beta)\|Y\|_\beta,
    \end{aligned}
\end{equation*}
hence  
\begin{equation}\label{eq:|TaY|<= M|a| + q|Y|}
    \|\mathcal{T}_a(Y)\|_\beta \leq M\|a\|_{\mathscr{T}_0} + q(1,R,\beta)\|Y\|_\beta.
\end{equation}
Because $q(1,R,\beta)<1$, $\|a\|<\varrho(\beta,R)$ and $\|Y\|_\beta\leq R$, the
right-hand side is strictly smaller than $R$.
 Moreover, by the mean-value
formula in the Banach space $\mathscr X_\beta$ and the preceding derivative
estimate in \Cref{lem:estimate of D^n_YT_a(Y)},
\[
\|\mathcal T_a(Y_1)-\mathcal T_a(Y_2)\|_\beta
\leq q(1,R,\beta)\|Y_1-Y_2\|_\beta
\]
for all $Y_1,Y_2\in\overline B_R^\beta$.

\end{proof}

Since $\overline B_R^\beta$ is complete, the contraction mapping theorem
gives a unique fixed point $Y_a\in\overline B_R^\beta$.  The strict
invariance just proved implies in fact that $Y_a\in B_R^\beta$.  In
particular, $Y_0=0$.

\begin{proposition}\label{prop:Y_a is a solution of the system in new coordinates}
For every $a\in B_{\varrho(\beta,R)}$, the fixed point $Y_a$ solves
\eqref{eq:McGehee system under coordinates chi} on $[0,+\infty)$ and
satisfies
\[
\|Y_a(\tau)\|\leq R e^{-\beta\tau}.
\]
\end{proposition}
\begin{proof}
Write the stable and center--unstable integral terms in the fixed-point
formula as $I_{ss}(\tau)$ and $I_{cu}(\tau)$, respectively. 
Hence
\begin{equation*}
    Y_a(\tau) = \mathrm{e}^{A\tau}a + I_{ss}(\tau) + I_{cu}(\tau).
\end{equation*}
The preceding
decay estimates justify differentiation and give
\[
I_{ss}'(\tau) = P_{ss}N(Y_a(\tau)) + AI_{ss}(\tau),
\qquad
I_{cu}'(\tau) = P_{cu}N(Y_a(\tau)) + AI_{cu}(\tau).
\]
Differentiating
$Y_a=e^{A\tau}a+I_{ss}+I_{cu}$ therefore yields
\[
Y_a'=AY_a+N(Y_a)=f(Y_a).
\]
The decay estimate follows directly from
$Y_a\in\overline B_R^\beta$; in fact the strict invariance of the ball gives
$Y_a\in B_R^\beta$.
\end{proof}

Evaluating the fixed-point equation at $\tau=0$ gives
\[
P_{ss}Y_a(0)=a,
\qquad
P_{cu}Y_a(0)
=-\int_0^{+\infty}e^{-A\sigma}P_{cu}N(Y_a(\sigma))\,d\sigma.
\]
Define
\[
h(a):=P_{cu}Y_a(0)
\]
and
\[
\widetilde L(\beta,R)
:=\{a+h(a):a\in B_{\varrho(\beta,R)}\}.
\]
Hence, $Y_a(0) = a + h(a)$.
Since $Y_a(0)\in B_R\subset\chi(\mathcal O)$,
$\widetilde L(\beta,R)\subset\chi(\mathcal O)$.

\begin{lemma}\label{lem:regularity lemma for submanifold arguement for LR}
The fixed-point map
\[
a\longmapsto Y_a\in\mathscr X_\beta
\]
and the graph map $h$ are $C^k$ on $B_{\varrho(\beta,R)}$.  Moreover,
$h(0)=0$ and $Dh(0)=0$.
\end{lemma}

\begin{proof}
On the open set $B_{\varrho(\beta,R)}\times B_R^\beta$, let
\[
\mathcal F(a,Y):=Y-\mathcal T_a(Y).
\]
By \Cref{lem:estimate of D^n_YT_a(Y)}, $\mathcal F$ is $C^k$.  At a fixed
point,
\[
D_Y\mathcal F(a,Y_a)=I-D_Y\mathcal T_a(Y_a),
\]
and
\[
\|D_Y\mathcal T_a(Y_a)\|\leq q(1,R,\beta)<1.
\]
Thus $D_Y\mathcal F(a,Y_a)$ is invertible, with inverse given by the
Neumann series.  The implicit function theorem and uniqueness of the fixed
point show that $a\mapsto Y_a$ is $C^k$; hence so is
$h(a)=P_{cu}Y_a(0)$.  Finally, $Y_0=0$ gives $h(0)=0$, while $DN(0)=0$
and the integral formula for $h$ give $Dh(0)=0$.
\end{proof}

\begin{proposition}
The set $\widetilde L(\beta,R)$ is an $m$-dimensional $C^k$ embedded
submanifold of $\mathscr T_0$, and
\[
T_0\widetilde L(\beta,R)=E_c^{ss}.
\]
\end{proposition}

\begin{proof}
Recall that $\mathscr T_0=E_c^{ss}\oplus E_c^{cu}$.  On the open set
$P_{ss}^{-1}(B_{\varrho(\beta,R)})$, define
\[
\phi(u)
:=\bigl(P_{ss}u,\,P_{cu}u-h(P_{ss}u)\bigr).
\]
This is a $C^k$ diffeomorphism onto
$B_{\varrho(\beta,R)}\times E_c^{cu}$, with inverse
\[
\phi^{-1}(a,b)=a+h(a)+b.
\]
Moreover,
\[
\phi\bigl(\widetilde L(\beta,R)\bigr)
=B_{\varrho(\beta,R)}\times\{0\}.
\]
Hence $\widetilde L(\beta,R)$ is an $m$-dimensional embedded submanifold.
The parametrization $F(a)=a+h(a)$ satisfies
\[
DF(0)=I+Dh(0)=I\quad\text{on }E_c^{ss},
\]
which proves $T_0\widetilde L(\beta,R)=E_c^{ss}$.
\end{proof}

The next subsection identifies this Lyapunov--Perron graph with the local
strong stable set in the physical zero-energy phase space.

\subsection{Proof of \Cref{thm:strong stable set is a submfd and its tangent space is Ecss}}\label{subscn: main part of the proof of thm:strong stable set is a submfd and its tangent space is Ecss}
We begin with Duhamel's formula.
\begin{lemma}\label{lem:Duhamel's Formula}
If $Y$ solves
\[
Y'(\tau)=AY(\tau)+ b(\tau),
\]
then, for every $\tau_1\leq\tau_2$,
\[
Y(\tau_2)=e^{A(\tau_2-\tau_1)}Y(\tau_1)
+\int_{\tau_1}^{\tau_2}e^{A(\tau_2-\sigma)} b(\sigma)\,d\sigma.
\]
\end{lemma}

See \proofref{lem:Duhamel's Formula}{apdx:1}.

\begin{proposition}\label{prop:tilde L_R = all strong stable orbits}
For every $R>0$ such that $q(1,R,\beta)<1$,
\[
\widetilde L(\beta,R)
=
\left\{
Y(0):
\begin{array}{l}
Y'=f(Y),\quad \|Y\|_\beta<R,\\[1mm]
\|P_{ss}Y(0)\|<\varrho(\beta,R)
\end{array}
\right\}.
\]
\end{proposition}

\begin{proof}
The inclusion from left to right follows from
\Cref{prop:T_a is a contraction,prop:Y_a is a solution of the system in new coordinates}.
Conversely, let $Y\in B_R^\beta$ solve $Y'=AY+N(Y)$ and put
$a=P_{ss}Y(0)\in B_{\varrho(\beta,R)}$.  Duhamel's formula gives
\begin{equation}\label{eq:proof of all strong stable orbits are in tilde L_R, ss part}
P_{ss}Y(\tau)
=e^{A\tau}a
+\int_0^\tau e^{A(\tau-\sigma)}P_{ss}N(Y(\sigma))\,d\sigma.
\end{equation}
For $T>\tau$ it also gives
\[
P_{cu}Y(T)
=e^{A(T-\tau)}P_{cu}Y(\tau)
+\int_\tau^T e^{A(T-\sigma)}P_{cu}N(Y(\sigma))\,d\sigma.
\]
Multiplying by $e^{-A(T-\tau)}P_{cu}$ yields
\[
P_{cu}Y(\tau)
=e^{-A(T-\tau)}P_{cu}Y(T)
-\int_\tau^T e^{A(\tau-\sigma)}P_{cu}N(Y(\sigma))\,d\sigma.
\]
Since $\|e^{-A(T-\tau)}P_{cu}\|\leq M$ and $Y(T)\to0$, letting
$T\to+\infty$ gives
\begin{equation}\label{eq:proof of all strong stable orbits are in tilde L_R, cu part}
P_{cu}Y(\tau)
=-\int_\tau^{+\infty}e^{A(\tau-\sigma)}P_{cu}N(Y(\sigma))\,d\sigma.
\end{equation}
Adding the \eqref{eq:proof of all strong stable orbits are in tilde L_R, ss part} shows that $Y=\mathcal T_a(Y)$.  The fixed point
is unique, hence $Y=Y_a$ and $Y(0)=a+h(a)\in\widetilde L(\beta,R)$.
\end{proof}

We next record the finite-time flow facts used below.  For $T>0$, let
\begin{align*}
\mathcal D_T^+
&:=\{Z\in\mathcal P:\Phi^\tau(Z)\text{ exists for }0\leq\tau\leq T\},\\
\mathcal D_T^-
&:=\{Z\in\mathcal P:\Phi^\tau(Z)\text{ exists for }-T\leq\tau\leq0\}.
\end{align*}
\begin{lemma}
For every $T>0$, the sets $\mathcal D_T^+$ and $\mathcal D_T^-$ are open
in $\mathcal P$, and
\[
\Phi^T:\mathcal D_T^+\longrightarrow\mathcal D_T^-
\]
is a $C^k$ diffeomorphism whose inverse is
$\Phi^{-T}:\mathcal D_T^-\to\mathcal D_T^+$.
\end{lemma}

\begin{proof}
For the smooth McGehee vector field on $\mathcal P$, the maximal flow domain
\[
\mathscr D:=\{(\tau,Z)\in\mathbb R\times\mathcal P:
\Phi^\tau(Z)\text{ is defined}\}
\]
is open.  Since existence at time $T$ implies existence on the whole segment
from $0$ to $T$, one has
\[
\mathcal D_T^+=\{Z:(T,Z)\in\mathscr D\},
\qquad
\mathcal D_T^-=\{Z:(-T,Z)\in\mathscr D\}.
\]
Both sets are therefore open.
Smooth dependence on initial data makes $\Phi^T$ smooth on
$\mathcal D_T^+$.  The flow property gives
\[
\Phi^{-T}\circ\Phi^T=\operatorname{id}_{\mathcal D_T^+},
\qquad
\Phi^T\circ\Phi^{-T}=\operatorname{id}_{\mathcal D_T^-},
\]
so $\Phi^T$ is a diffeomorphism with the stated inverse.
\end{proof}

We use the following continuous local size function relative
to the equilibrium:
\[
\mathfrak d_c(\rho,s,v,w)
:=\max\bigl\{|\rho|,d_{\mathcal E}(s,c),|v-v_*|,\|w\|\bigr\}.
\]
Thus, 
\begin{equation}\label{eq:Uc by local size function}
\mathcal U_c(\delta)
=
\{Z\in\mathcal P:\mathfrak d_c(Z)<\delta,\ \rho\geq0,\ v\leq v_*\}.
\end{equation}
For $T>0$, set
\[
\mathcal O_T(\delta)
:=\left\{Z\in\mathcal D_T^+:
\mathfrak d_c(\Phi^\tau(Z))<\delta
\text{ for all }0\leq\tau\leq T\right\}.
\]
$\mathcal O_T(\delta)$ is open in $\mathcal P$ since the map
\[
Z\longmapsto
\max_{0\leq\tau\leq T}\mathfrak d_c(\Phi^\tau(Z))
\]
is continuous on $\mathcal D_T^+$ by continuous dependence of the flow on
initial data, uniformly on the compact interval $[0,T]$.  Hence
$\mathcal O_T(\delta)$ is open in the open set $\mathcal D_T^+$, and
therefore in $\mathcal P$.

We now prove the \Cref{thm:strong stable set is a submfd and its tangent space is Ecss}.
Let
\[
\widetilde{\mathcal U}_R^{ss}
:=\chi^{-1}(\widetilde L(\beta,R)).
\]
It is an $m$-dimensional $C^k$ embedded submanifold of $\mathcal{P}$ near $e_c^+$.
\begin{lemma}\label{lem:tangent-space identification at ec+}
Under the natural identification
$T_{e_c^+}\mathcal P\cong\mathscr T_0$, one has
\[
d\chi|_{e_c^+}=d\chi^{-1}|_0=I.
\]
\end{lemma}

\begin{proof}
Put $\varphi=\chi^{-1}$.  In the product coordinates,
\[
 \varphi(\rho,\xi,W)=
 \Bigl(\rho,\exp_c\xi,
 \widehat V(\xi,W),P_\xi W\Bigr),
 \qquad
 \widehat V(\xi,W)=\sqrt{2U(\exp_c\xi)-\|W\|^2}.
\]
The radicand is positive near $(0,0)$ because its value there is $v_*^2$.
Since $(d\exp_c)_0=P_0=\id$, $d(U|_{\mathcal E})_c=0$, and
$d(\|W\|^2)_0=0$, one has $d\widehat V_{(0,0)}=0$.  Moreover,
\[
 d(P_\xi W)_{(0,0)}[\eta,\zeta]
 =(dP)_0[\eta]0+P_0\zeta=\zeta.
\]
Therefore
\[
 d\varphi_0(a,\eta,\zeta)=(a,\eta,0,\zeta).
\]
Under the natural identification
$\mathcal J(a,\eta,\zeta)=(a,\eta,0,\zeta)$ of the coordinate tangent
space with $T_{e_c^+}\mathcal P$, the chain rule applied to
$\chi\circ\varphi=\id$ gives
\[
 d\chi|_{e_c^+}=d\chi^{-1}|_0=\id.
\]
\end{proof}

Consequently,
\[
T_{e_c^+}\widetilde{\mathcal U}_R^{ss}=d\chi^{-1}|_0 T_0 \tilde{L}(\beta,R) =  E_c^{ss}.
\]
The sets $\{\rho<0\}$, $\{\rho=0\}$, and $\{\rho>0\}$ are invariant
because $\rho'=-v\rho/2$.  Define the physical part
\[
\mathcal U_R^{ss}
:=\widetilde{\mathcal U}_R^{ss}\cap\{\rho\geq0\}.
\]
Since the center--unstable space has zero $\rho$-component,
$\rho(a+h(a))=\rho(a)$.  Hence we obtains
the Lyapunov--Perron parametrization
\begin{equation}\label{eq:parametrization of U_R^ss}
    \mathcal U_R^{ss}
    =
    \chi^{-1}\left(
    \{a+h(a):a\in B_{\varrho(\beta,R)}\subset E_c^{ss},\ \rho(a)\geq0\}
    \right).
\end{equation}
Thus $\mathcal U_R^{ss}$ is an $m$-dimensional $C^k$ embedded
submanifold with boundary $\rho=0$, and
\[
T_{e_c^+}\mathcal U_R^{ss}=E_c^{ss}.
\]

After decreasing $R$ if necessary, smoothness of $\chi^{-1}$ gives a
constant $C_\chi(R)$ such that
\begin{equation}\label{eq:local size controlled by chi coordinates}
\mathfrak d_c(\chi^{-1}(Y))
\leq C_\chi(R)\|Y\|_{\mathscr T_0},
\qquad \|Y\|_{\mathscr T_0}\leq R.
\end{equation}

\begin{lemma}\label{lem:conjucation of flows}
For $Z_a=\chi^{-1}(a+h(a))\in\mathcal U_R^{ss}$,
\[
\chi(\Phi^\tau(Z_a))=Y_a(\tau),
\qquad \tau\geq0.
\]
\end{lemma}
See \proofref{lem:conjucation of flows}{apdx:1}.

Choose $R$ so that, in addition,
$R C_\chi(R)<\delta$. Then using \Cref{lem:conjucation of flows} and$\|Y_a(\tau)\|_{\mathscr{T}_0}\leq R\mathrm{e}^{-\beta\tau}$ gives
    \begin{align*}
        \mathfrak d_c(\Phi^\tau(Z_a)) & =  \mathfrak d_c(\chi^{-1}(Y_a(\tau)))\\
        &\leq C_\chi(R) \|Y_a(\tau)\|_{\mathscr{T}_0}\\
        &\leq RC_\chi(R) \mathrm{e}^{-\beta\tau}\\
        &<\delta \mathrm{e}^{-\beta\tau}.
    \end{align*}
Moreover, $\rho(\Phi^\tau(Z_a))\geq0$.  Since the orbit tends to
$e_c^+$ and $v'=\|w\|^2/2\geq0$, one also has
$v(\Phi^\tau(Z_a))\leq v_*$ for every $\tau\geq0$.  By
\eqref{eq:Uc by local size function}, the whole orbit lies in
$\mathcal U_c(\delta)$.
We have proved the following result.
\begin{proposition}\label{prop:results of step 1 of proving W(delta) is submfd}
For $R>0$ sufficiently small, $\mathcal U_R^{ss}$ is an
$m$-dimensional $C^k$ embedded submanifold with boundary $\rho=0$,
\[
T_{e_c^+}\mathcal U_R^{ss}=E_c^{ss},
\]
and
\[
\Phi^\tau(\mathcal U_R^{ss})\subset\mathcal U_c(\delta),
\qquad
\Phi^\tau(Z)\to e_c^+
\quad(\tau\to+\infty)
\]
for every $Z\in\mathcal U_R^{ss}$.  In particular,
\[
\mathcal U_R^{ss}\subset W_{loc}^{ss}(e_c^+;\delta).
\]
\end{proposition}

\begin{proposition}\label{prop:results of step 2 of proving W(delta) is submfd}
There exists $T_*>0$, independent of
$Z\in W_{loc}^{ss}(e_c^+;\delta)$, such that
\begin{equation}\label{eq:uniform shrinking}
\Phi^{T_*}\bigl(W_{loc}^{ss}(e_c^+;\delta)\bigr)
\subset\mathcal U_R^{ss}.
\end{equation}
\end{proposition}

\begin{proof}
By \Cref{prop:outgoing zero-energy solutions locally exponentially converge},
there are constants $C>0$ and $\gamma_\delta\ge\beta$, independent of the
orbit, such that
\begin{equation}\label{eq:uniform estimate of strong stable flow}
\mathfrak d_c(\Phi^\tau(Z))
\leq C e^{-\gamma_\delta\tau}
\end{equation}
for $Z\in W_{loc}^{ss}(e_c^+;\delta)$ and $\tau\geq0$.
Choose a neighborhood on which $\chi$ is defined and a constant $K_\chi$
such that
\begin{equation}\label{eq:|chi Z| leq K d(Z,ec+)}
\|\chi(Z)\|\leq K_\chi\mathfrak d_c(Z).
\end{equation}
Taking $T_*$ sufficiently large ensures that the whole tail
$\Phi^{[T_*,+\infty)}(Z)$ lies in this coordinate neighborhood and that
\[
K_\chi C e^{-\gamma_\delta T_*}<R,
\qquad
\|P_{ss}\|K_\chi C e^{-\gamma_\delta T_*}
<\varrho(\beta,R).
\]
For such $T_*$, define
\[
Y_Z(\tau):=\chi(\Phi^{\tau+T_*}(Z)),
\qquad \tau\geq0.
\]
Then 
\begin{align*}
    Y'_Z(\tau) &= \frac{d}{d\tau}\chi\left(\Phi^{\tau+T_*}(Z)\right)\\
               &= d\chi|_{\Phi^{\tau+T_*}(Z)}\dfrac{d}{d\tau} \Phi^{\tau+T_*}(Z)\\
               &= d\chi|_{\Phi^{\tau+T_*}(Z)}\mathfrak{X}\left(\Phi^{\tau+T_*}(Z)\right)\\
               &= d\chi|_{\chi^{-1}(Y_Z(\tau))}\mathfrak{X}\left[\chi^{-1}\left(Y_Z(\tau)\right)\right]\\
               &=f\left(Y_Z(\tau)\right),
\end{align*}
The last equality follows from \eqref{eq:f is the pushforward}.
    We estimate $\|Y_Z\|_\beta$ and $\|P_{ss}Y_Z(0)\|$ as follows.
\begin{equation*}
    \begin{aligned}
        \|Y_Z\|_\beta & = \sup_{\tau\geq0}\mathrm{e}^{\beta\tau}\|Y_Z(\tau)\|\\
        &\leq K_\chi C\mathrm{e}^{-\gamma_\delta T_*} \sup_{\tau\geq0} \mathrm{e}^{-(\gamma_\delta - \beta) \tau}\\
        &\leq K_\chi C\mathrm{e}^{-\gamma_\delta T_*}\\
        &<R,
    \end{aligned}
\end{equation*}
and 
\begin{equation*}
    \|P_{ss}Y_Z(0)\| \leq \|P_{ss}\| K_\chi C\mathrm{e}^{-\gamma_\delta T_*} < \varrho(\beta,R).
\end{equation*}
By \Cref{prop:tilde L_R = all strong stable orbits},
$Y_Z(0)\in\widetilde L(\beta,R)$.  Since $Z$ is physical,
$\rho(\Phi^{T_*}(Z))\geq0$, and therefore
$\Phi^{T_*}(Z)\in\mathcal U_R^{ss}$.
\end{proof}

\begin{proposition}\label{prop:results of step 3 of proving W(delta) is submfd}
With $T_*$ as above,
\begin{equation}\label{eq:characterization of W_R^ss}
W_{loc}^{ss}(e_c^+;\delta)
=
\mathcal O_{T_*}(\delta)
\cap
\Phi^{-T_*}\bigl(\mathcal U_R^{ss}\cap\mathcal D_{T_*}^-\bigr).
\end{equation}
\end{proposition}

\begin{proof}
If $Z\in W_{loc}^{ss}(e_c^+;\delta)$, then
$\Phi^\tau(Z)\in\mathcal U_c(\delta)$ for every $\tau\geq0$; hence
$Z\in\mathcal O_{T_*}(\delta)$.  \Cref{prop:results of step 2 of proving W(delta) is submfd} gives
$\Phi^{T_*}(Z)\in\mathcal U_R^{ss}$, proving the first inclusion.

Conversely, let $Z$ belong to the right-hand side of
\eqref{eq:characterization of W_R^ss}.  On $[0,T_*]$ the definition of
$\mathcal O_{T_*}(\delta)$ gives
$\mathfrak d_c(\Phi^\tau(Z))<\delta$.  Since
$\Phi^{T_*}(Z)\in\mathcal U_R^{ss}$, its $\rho$-coordinate is
nonnegative.  The sign of $\rho$ is invariant, so
$\rho(\Phi^\tau(Z))\geq0$ for $0\leq\tau\leq T_*$.  Also the future
orbit of $\Phi^{T_*}(Z)$ tends to $e_c^+$, and $v'=\|w\|^2/2\geq0$;
therefore $v(\Phi^{T_*}(Z))\leq v_*$.  Monotonicity then yields
$v(\Phi^\tau(Z))\leq v_*$ for $0\leq\tau\leq T_*$.  Thus
\eqref{eq:Uc by local size function} implies
\[
\Phi^\tau(Z)\in\mathcal U_c(\delta),
\qquad 0\leq\tau\leq T_*.
\]
For $\tau\geq T_*$, Step~1 applied to $\Phi^{T_*}(Z)$ gives the same
conclusion and convergence to $e_c^+$.  Hence
$Z\in W_{loc}^{ss}(e_c^+;\delta)$.
\end{proof}

In this step we first need a lemma.
\begin{lemma}\label{lem:D Phi^tau = e^{Atau}}
For every $\tau\in\mathbb R$ for which the flow is defined at the
equilibrium,
\[
D\Phi^\tau_{e_c^+}=e^{A\tau}.
\]\end{lemma}
See \proofref{lem:D Phi^tau = e^{Atau}}{apdx:1}.

\begin{proof}
    
The set $\mathcal U_R^{ss}\cap\mathcal D_{T_*}^-$ is an
$m$-dimensional $C^k$ embedded submanifold with boundary of $\mathcal P$.
Since $\Phi^{-T_*}:\mathcal D_{T_*}^-\to\mathcal D_{T_*}^+$ is a
$C^k$ diffeomorphism, its inverse image is a $C^k$ embedded submanifold
with boundary.  Intersecting with the open set
$\mathcal O_{T_*}(\delta)$ and using
\eqref{eq:characterization of W_R^ss} proves that
$W_{loc}^{ss}(e_c^+;\delta)$ is an $m$-dimensional $C^k$ embedded
submanifold with boundary.

Both $\mathcal O_{T_*}(\delta)$ and $\mathcal D_{T_*}^-$ contain
$e_c^+$.  Therefore, using the results in
 \Cref{lem:D Phi^tau = e^{Atau}} and
  \Cref{prop:results of step 1 of proving W(delta) is submfd},
\begin{align*}
    T_{e_c^+} W_{loc}^{ss}(e_c^+;\delta) &= T_{e_c^+} \Phi^{-T_*}(\mathcal{U}_R^{ss})
    =D \Phi^{-T_*}_{e_c^+} T_{e_c^+} \mathcal{U}_R^{ss}\\
    &= \mathrm{e}^{-A T_*} E_c^{ss}
    =E_c^{ss}.
\end{align*}
The last equality follows from the $A$-invariance of $E_c^{ss}$.

\end{proof}

\section{Cone uniqueness}
\label{secn:cone-uniqueness}

\subsection{The graph representation and cone-confined uniqueness}

Based on the linearization at $e_c^+$ decomposition constructed in \Cref{thm:strong stable set is a submfd and its tangent space is Ecss}, 
we apply the stable manifold theorem to claim the following important local diffeomorphism.
Write $\mathcal{N} = [0,+\infty)\times \mathcal{E}_0$.
Consider the natural projection
\begin{align*}
        \pi:\mathcal{P}_0&\rightarrow \mathcal{N}\\
        (\rho,s,v,w)&\mapsto (\rho,s).
\end{align*}
Write 
$
W(\delta):=W_{\mathrm{loc}}^{ss}(e_c^+;\delta)
$
for brevity in this subsection.

The local inverse mapping theorem provides a graph over a relatively open neighborhood of
$(0,c)$. The additional inclusion below supplies the
injectivity needed for the subsequent uniqueness argument.
Given any $\delta\in(0,\delta_0)$, we have the following proposition.
\begin{proposition}[Local strong stable graph over a configuration neighborhood]
\label{prop:local graph near ec+}

There exist $\eta\in(0,\delta)$, and relatively open neighborhoods $\mathcal{W}$ 
of $e_c^+$ in $W(\delta)$ and $\mathcal{B}$ of $(0,c)$ in $\mathcal{N}$,
 such that the following statements hold.

\begin{enumerate}
        \item The restriction $\pi|_{\mathcal W}:\mathcal W\to\mathcal{B}$ is a $C^\infty$
diffeomorphism. 
\item There exists some $C^\infty$
map $G:\mathcal{B}\to X$ and $G(\rho,s)\in T_s\mathcal{E}_0$ such that
 \begin{equation}
\label{eq:rectangular-stable-graph-lift}
(\pi|_\mathcal{W})^{-1}(\rho,s)
=
\bigl(\rho,s,V(s,G(\rho,s)),G(\rho,s)\bigr),~~(\rho,s)\in\mathcal{B}.
\end{equation} 
Hence, $\mathcal{W}$
satisfies the graph representation
\begin{equation}
\label{eq:rectangular-stable-graph}
\mathcal{W}
=\bigl\{
(\rho,s,v,w):
(\rho,s)\in\mathcal{B},~
 w=G(\rho,s),~ v=V(s,G(\rho,s))
\bigr\},
\end{equation}

\item \begin{equation}
\label{eq:small-stable-plaque-in-injective-neighborhood}
W(\eta)\subset\mathcal W.
\end{equation}
In particular, $\pi$ is injective on
$W(\eta)$.
\end{enumerate}

\end{proposition}

\begin{proof}
Based on \Cref{thm:strong stable set is a submfd and its tangent space is Ecss}, the differential of $\pi$ at $e_c^+$ is
\begin{align*}
        d\pi|_{e_c^+}: E_c^{ss}&\rightarrow\mathbb{R}\oplus T_c\mathcal{E}\\
                   (\rho,\xi,w)&\mapsto (\rho,\xi).
\end{align*}
We claim that $d\pi|_{e_c^+}$ is a linear homomorphism. The linearity and surjection is obvious. 
Given arbitrary $(\rho,\xi)= 0$,
then $\Lambda^-\xi = 0$, the bijection follows.

Both $W(\delta)$ and $\mathcal{N}$ have dimension
$m$. Near $e_c^+$, the boundary of $W(\delta)$ is contained in
$\{\rho=0\}$, and the projection sends this boundary into
$\{\rho=0\}\times\mathcal E_0$. The inverse mapping theorem for manifolds with
boundary gives simply connected neighborhoods
$\mathcal W$ of $e_c^+$ and neighborhoods $\mathcal B$ of $(0,c)$, which are relatively open in $W(\delta)$ and $\mathcal{N}$ respectively,
 such that
$$
\pi|_{\mathcal W}:\mathcal W\longrightarrow\mathcal B
$$
is a $C^\infty$ diffeomorphism; see
\cite[Theorem~2.2.6]{MargalefRoigOuterelo1992}.

Since $\pi|_{\mathcal W}$ is a diffeomorphism.
The identity $\pi\circ \pi^{-1}|_{\mathcal{B}}=\id_{\mathcal{B}}$ implies that there are
$C^k$ component maps $\widehat v$ and $G$ with
$$
(\pi|_\mathcal{W})^{-1}(\rho,s)
=\bigl(\rho,s,\widehat v(\rho,s),G(\rho,s)\bigr),
\quad G(\rho,s)\in T_s\mathcal{E}, \quad
(\rho,s)\in \mathcal{B}.
$$
Because $(\pi|_{\mathcal{W}})^{-1}(\rho,s)\in \mathcal{W}\subset W(\delta)\subset\mathcal U_c(\delta)$ as $(\rho,s)\in \mathcal{B}$, one has
$
\widehat v(\rho,s)>v_*-\delta>0.
$
The zero-energy relation thus yields
\begin{equation*}
\widehat v(\rho,s)
=\sqrt{2U(s)-\lVert G(\rho,s)\rVert^2}
=V(s,G(\rho,s)),\quad (\rho,s)\in\mathcal{B}.
\end{equation*}
This proves \eqref{eq:rectangular-stable-graph-lift}, \eqref{eq:rectangular-stable-graph} and
the asserted diffeomorphism for $\mathcal W$.

It remains to obtain the inclusion
\eqref{eq:small-stable-plaque-in-injective-neighborhood}. Since $W(\delta)$ is
embedded, its topology is the subspace topology inherited from
$\mathcal P$. 
Thus there is a open neighborhood $\mathcal A$
of $e_c^+$ in $\mathcal P$ such that
$$
\mathcal W = W(\delta)\cap\mathcal A.
$$
Since sets $\mathcal U_c(\eta)$ shrink to $e_c^+$ as $\eta\downarrow0$,
choose $\eta\in(0,\delta)$ so that
$$
\mathcal U_c(\eta)\subset\mathcal A.
$$
The defining forward-confinement condition immediately implies
$$
W(\eta)\subset W(\delta).
$$
Consequently,
\begin{equation*}
W(\eta)\subset W(\delta)\cap\mathcal U_c(\eta)\subset W(\delta)\cap\mathcal A
=\mathcal W.
\end{equation*}
Injectivity follows by restriction of the injective map
$\pi|_{\mathcal W}$.
\end{proof}

The local strong-stable graph provides, for every configuration sufficiently
close to the ray $\mathbb R_{+}c$ at large radius, a distinguished candidate
initial velocity.  The graph representation alone, however, does not imply
that the corresponding orbit remains in the same prescribed rectangular
configuration neighborhood: a stable orbit may converge to an equilibrium
after first leaving a smaller neighborhood.  Thus, in order to obtain the
existence assertion in the first cone theorem---rather than only uniqueness
among motions that are already assumed to remain in the cone---one must prove
positive invariance of a suitably truncated cone on the graph.  The
proof does not use action minimization and therefore establishes
the required existence statement in the class of general $c$-parabolic
motions, not only for parabolic geodesic rays.

\begin{proposition}[Positive invariance of a truncated cone on the strong-stable graph]
\label{prop:positive-invariance-strong-stable-graph}
There exist $R_*, \delta_* > 0$ such that, for every $x_0 = r_0 s_0 \in \mathcal{C}_c(R_*, \delta_*)$, the Newtonian solution determined by the strong-stable graph
 is defined for all positive physical time, is a collision-free $c$-parabolic motion, and remains in 
$\mathcal{C}_c(R_*, \delta_*)$ for all $t \ge 0.$
\end{proposition}

\begin{proof}

For all sufficiently small $\rho\geq0$, the point
$(\rho,c,v_*,0)$
lies on the local strong-stable manifold.  Indeed, its
physical oribt is the homothetic orbit $\left(r_0^{3/2}+\frac{3}{2}v_*t\right)^{2/3}c$ and the corresponding
forward McGehee orbit is $\bigl(\rho e^{-v_*\tau/2},c,v_*,0\bigr),~\tau\geq0$
which converges exponentially to $e_c^+=(0,c,v_*,0)$.  
The initial state of \(\rho\) itself lies in the local graph representation $\mathcal W$.  The uniqueness of the graph value
over $(\rho,c)$ therefore gives
\begin{equation}
 G(\rho,c)=0
 \text{ for all sufficiently small }\rho>0.
 \label{eq:G-rho-c-zero}
\end{equation}
In a normal neighborhood of $c$ in $\mathcal{E}_0$, remember
\(
 s=\exp_c\xi,
  \xi\in T_c\mathcal{E}_0,
\)
one has
\[
 d_\mathcal{E}(\exp_c\xi,c)=\lVert\xi\rVert.
\]
Since $s'=w=G(\rho,s)$, differenciating $s=\exp_c\xi$ gives 
\(d\exp_c|_\xi \xi' = G(\rho,\exp_c\xi)\).
Define the reduced shape vector field in these fixed coordinates by
\begin{equation}
 g(\rho,\xi)
 :=\bigl(d\exp_c|_\xi\bigr)^{-1}
       G\bigl(\rho,\exp_c\xi\bigr).
 \label{eq:def-reduced-shape-field}
\end{equation}
Since $\mathcal W$ is a plaque of the locally positively invariant
strong-stable manifold, the full McGehee vector field is tangent to the graph
for as long as the configuration projection remains in the graph domain.  
Then the McGehee equation
$s'=w$ becomes
\begin{equation}
 \xi'=g(\rho,\xi).
 \label{eq:reduced-shape-equation}
\end{equation}
Let
\(
\widehat{G}(\rho,\xi) = P_\xi^{-1} G(\rho, \exp_c \xi).
\)
The graph being tangent to \(E_c^{ss}\) gives \(D_\xi \widehat{G}(0,0) = \Lambda^-,  D_\rho \widehat{G}(0,0) = 0\). We have
\(B(0) = I \, , \; \widehat{G}(0,0) = 0\), 
and since \(g = B \widehat{G}\), then
\[
D_\xi g(0,0) = (DB)_0 \widehat{G}(0,0) + B(0) D_\xi \widehat{G}(0,0) = \Lambda^-.
\]
Equation \eqref{eq:G-rho-c-zero} implies
\begin{equation}
 g(\rho,0)=0.
 \label{eq:g-rho-zero}
\end{equation}

Remember the strong-stable shape operator is
\[
 \Lambda^-
 =-\frac{v_*}{4}I
  -\sqrt{\frac{v_*^2}{16}I+H_c}\le -\frac{v_*}{2}I.
\]
Notice that this remains strict even in directions belonging to
$\ker H_c$.
Therefore, by continuity, we can choose a sufficiently small rectangle
\[
Q = [0, \rho_*) \times \{\|\xi\| < \delta_*\}
\]
whose closure contained in the domain $\mathcal{B}$ of the graph, such that
\begin{equation*}
        v(\rho, \xi) \geq \frac{v_*}{2}, 
 \bigl\langle \xi,D_\xi g(\rho,u\xi)\xi\bigr\rangle
 \leq -\frac{v_*}{4}\lVert \xi\rVert^2  \text{ for all } (\rho,\xi)\in Q \text{ and } u\in[0,1].
\end{equation*}
Along a graph orbit that remains in the above coordinate neighborhood, using $g(\rho,0)=0$ and $\displaystyle g(\rho,\xi)=\int_0^1 D_\xi g(\rho,\theta\xi)\xi\,d\theta$,
we obtain 
\begin{equation} \label{eq:inward-shape-estimate}
 \frac{d}{d\tau}\frac12\lVert\xi\rVert^2 = \langle\xi,g(\rho,\xi)\rangle
 \leq -\frac{v_*}{4}\lVert\xi\rVert^2\quad\text{for all}~~ (\rho,\xi)\in Q.
\end{equation}
On the other hand, $\rho'\le -\dfrac{v_*}{4}\rho$, hence by the Gronwall's inequality,
\begin{equation}
         0<\rho(\tau)
 \leq\rho(0)e^{-\frac{v_*}{4}\tau}<\rho_*,\quad
  d_\mathcal{E}(s(\tau),c) = 
 \lVert\xi(\tau)\rVert
 \leq e^{-\frac{v_*}{4}\tau}\lVert\xi(0)\rVert<\delta_*\quad \text{for all}~~ \tau\ge0.
 \label{eq:positive-invariance-estimates}
\end{equation}
Therefore, for every forward time, the orbit remains in the rectangle 
for any initial configuration $x_0 = r_0s_0$ such that $\rho_0 = r_0^{-1/2}\in(0,\rho_*),~ d_\mathcal{E}(s_0, c)<\delta_*$.

\Cref{prop:Equivalence between McGehee and Newton,prop:McGehee orbit convergence to ea+ gives a parabolic motion}
show that the reconstructed physical curve is a global collision-free
zero-energy $c$-parabolic solution.
Set $R_*:=\rho_*^{-2}$.  If
$x_0=r_0s_0\in\mathcal C_c(R_*,\delta_*)$, then
$\rho_0=r_0^{-1/2}<\rho_*$ and
$d_{\mathcal E}(s_0,c)<\delta_*$.  Formula
\eqref{eq:positive-invariance-estimates} gives
$r(t)>R_*$ and $d_{\mathcal E}(s(t),c)<\delta_*$ for every $t\geq0$.
As desired.
\end{proof}

\begin{remark}[The inward-pointing condition]
\label{rem:inward-pointing-condition}
Estimate~\eqref{eq:inward-shape-estimate} has a direct boundary
interpretation.  On the shape boundary
$\lVert\xi\rVert=\delta_*$ one has
\[
 \frac{d}{d\tau}\frac12\lVert\xi\rVert^2
 =\langle\xi,g(\rho,\xi)\rangle
 \leq -\frac{v_*}{4}\delta_*^2<0.
\]
Thus the reduced shape vector field is strictly transverse to that boundary
and points toward the interior; in particular, the shape distance from
$c$ cannot increase there.  Likewise, on the outer radial boundary
$\rho=\rho_*$,
\[
 \rho'=-\frac12\rho_*v
 \leq-\frac{v_*}{4}\rho_*<0.
\]
These two strict inequalities are precisely the inward-pointing conditions
that rule out a first exit from the radial-shape rectangle.
\end{remark}

For uniqueness, one also needs an arbitrary $c$-parabolic motion that is already
assumed to remain in a sufficiently remote and sufficiently narrow
configuration cone.  Its entire forward McGehee lift must remain in the
prescribed local phase-space neighborhood on which the strong-stable graph
and the injectivity of the configuration projection are available.  The argument is purely dynamical.

\begin{proposition}[Configuration confinement implies phase-space confinement]
\label{prop:configuration-phase-confinement}
Let
$\gamma(t)=r(t)s(t)$ be a collision-free zero-energy $c$-parabolic motion,
and let
$$
Z(\tau)=(\rho(\tau),s(\tau),v(\tau),w(\tau))
$$
be its McGehee lift. Given positive numbers,
$\delta_\rho$, $\delta_s$, $\delta_v$, and $\delta_w$ that less than $v_*$, there exist
$R>0$ and $\delta>0$ such that the following implication holds: if
$$
r(t)>R,
\qquad
d_{\mathcal E}(s(t),c)<\delta
\qquad\text{for every }t\geq0,
$$
then
\begin{equation}
\label{eq:phase-confinement-conclusion}
0\leq\rho(\tau)<\delta_\rho,
\qquad
d_{\mathcal E}(s(\tau),c)<\delta_s,
\qquad
0\le v_* - v(\tau)<\delta_v,
\qquad
\lVert w(\tau)\rVert<\delta_w
\end{equation}
for every $\tau\geq0$.
\end{proposition}

\begin{proof}
We in this proof regard $s$ and $w$ as
$X$-valued curves and distinguish the
ordinary derivative $w'$ in $X$ from the covariant derivative $D_\tau w$
on $\mathcal E$.

Keep the prescribed tolerances $\delta_\rho,\delta_s,\delta_v,\delta_w$
fixed, and introduce the smaller auxiliary tolerance
\[
\widehat\delta_w
:=\frac12\min\left\{\delta_w,\frac{v_*}{2},\sqrt{v_*\delta_v}\right\}.
\]
Note that $U$ is smooth on a neighborhood of $B_\mathcal{E}(c,\delta_0)$. Set
\[
K:=\overline{B_\mathcal{E}(c,\delta_0)}
\quad
\overline U:=\max_{s\in K}U(s),
\quad
M:=\sqrt{2\overline U},
\quad
G_0:=\max_{s\in K}\|\nabla_{\mathcal E}U(s)\|,
\quad
L:=1+G_0+\frac32M^2.
\]
These constants depend only on the fixed neighborhood of $c$, not on the
motion under consideration. Put
\[
h:=\frac{\widehat\delta_w}{L}>0,
\]
and choose
\begin{equation}
\label{eq:configuration-aperture-choice}
0<\delta<\min\left\{\delta_s,\delta_0,
\frac{\widehat\delta_w^2}{4L}\right\}.
\end{equation}
Finally, choose
\begin{equation}
\label{eq:radial-threshold-choice}
R>\delta_\rho^{-2}.
\end{equation}

Let $\gamma(t)=r(t)s(t)$ satisfy the assumptions of the proposition and
remain in the configuration cone determined by $R$ and $\delta$ for every
$t\geq0$. 
The zero-energy relation yields, for every $\tau\geq0$,
\[
v(\tau)^2+\|w(\tau)\|^2=2U(s(\tau))\leq M^2,
\qquad
|v(\tau)|\leq M,
\qquad
\|w(\tau)\|\leq M.
\]
Recall from \eqref{eq:tangential projection of w'} that
\[
D_\tau w=w'+\|w\|^2s.
\]
Using the McGehee equation for $D_\tau w$, we therefore obtain the
ordinary, $X$-valued acceleration formula
\[
s''=w'
=\nabla_{\mathcal E}U(s)-\frac12vw-\|w\|^2s.
\]
In particular, since $\|s\|=1$,
\begin{equation}
\label{eq:confinement-ambient-acceleration}
\|s''(\tau)\|=\|w'(\tau)\|
\leq G_0+\frac12|v(\tau)|\|w(\tau)\|+\|w(\tau)\|^2
\leq G_0+\frac32M^2<L.
\end{equation}
This bound is uniform over all motions satisfying the configuration
confinement assumption.

We claim that
\begin{equation}
\label{eq:w-uniform-smallness}
\|w(\tau)\|<\widehat\delta_w<\delta_w
\qquad\text{for every }\tau\geq0.
\end{equation}
Suppose instead that $\|w(\tau_0)\|\geq\widehat\delta_w$ at some
$\tau_0\geq0$. Taylor's formula with integral remainder, applied to the
$X$-valued curve $s$, gives
\[
s(\tau_0+h)-s(\tau_0)
=h\,w(\tau_0)+\int_0^h(h-u)w'(\tau_0+u)\,du.
\]
By \eqref{eq:confinement-ambient-acceleration} and the choice of $h$,
\begin{equation}
\label{eq:confinement-forced-displacement}
\begin{aligned}
\|s(\tau_0+h)-s(\tau_0)\|
&\geq h\|w(\tau_0)\|-
\int_0^h(h-u)\|w'(\tau_0+u)\|\,du\\
&\geq h\widehat\delta_w-\frac12Lh^2
=\frac{\widehat\delta_w^2}{2L}.
\end{aligned}
\end{equation}
On the other hand, the ambient chord length is no greater than the
spherical distance. Hence configuration confinement implies
\[
\begin{aligned}
\|s(\tau_0+h)-s(\tau_0)\|
&\leq\|s(\tau_0+h)-c\|+\|s(\tau_0)-c\|\\
&\leq d_{\mathcal E}(s(\tau_0+h),c)
+d_{\mathcal E}(s(\tau_0),c)\\
&<2\delta<\frac{\widehat\delta_w^2}{2L},
\end{aligned}
\]
contradicting \eqref{eq:confinement-forced-displacement}. This proves
\eqref{eq:w-uniform-smallness}.

Combining the zero
energy constraint, minimality of $c$ and \eqref{eq:w-uniform-smallness} gives
\[
v(\tau)^2
=2U(s(\tau))-\|w(\tau)\|^2
>v_*^2-\widehat\delta_w^2>0.
\]
Thus $v(\tau)>0$ for all $\tau\geq0$, otherwise
\[
\dot r(t)=r(t)^{-1/2}v(\tau(t))<0
\qquad\text{for every }t\geq0,
\]
which would imply $r(t)\leq r(0)$, contradicting the $c$-parabolic condition
$r(t)\to+\infty$.  Consequently,
\[
\begin{aligned}
0\leq v_*-v(\tau)
&=\frac{v_*^2-v(\tau)^2}{v_*+v(\tau)}\\
&=\frac{\|w(\tau)\|^2-2\bigl(U(s(\tau))-U(c)\bigr)}
{v_*+v(\tau)}\\
&\leq\frac{\|w(\tau)\|^2}{v_*}
<\frac{\widehat\delta_w^2}{v_*}\le \delta_v.
\end{aligned}
\]
Finally, since $\dot{r} = r^{-1/2} v >0$,  the assumed radial and shape confinement and the choices of
$R$ and $\delta$ imply
\[
0<\rho(\tau)=r(t(\tau))^{-1/2}<R^{-1/2}<\delta_\rho,
\qquad
d_{\mathcal E}(s(\tau),c)<\delta<\delta_s.
\]
Together with \eqref{eq:w-uniform-smallness} and the bound for $v$, these
inequalities establish \eqref{eq:phase-confinement-conclusion} for every
$\tau\geq0$.
\end{proof}

We can now prove \Cref{thm:uniqueness-cone-confined-parabolic-motion}.
\begin{deferredproof}{thm:uniqueness-cone-confined-parabolic-motion}
    We have proved in \Cref{prop:positive-invariance-strong-stable-graph}
the existence and invariance assertions; the existing graph-injectivity
argument in \Cref{prop:local graph near ec+} proves uniqueness as follows.

 Retain $\eta\in(0,v_*)$ in \Cref{prop:local graph near ec+} such that $\pi$ is injective on $W(\eta)$.
     Then, by \Cref{prop:configuration-phase-confinement}, 
     increase $R_*$ and decrease $\delta_*$ in \Cref{prop:positive-invariance-strong-stable-graph}
     such that for all $c$-parabolic 
    motions whose radials and shapes confined by $\mathcal{C}_c(R_*,\delta_*)$, their McGehee lifts
    stay in $W(\eta)$.
    Given any $x_0 = r_0s_0\in \mathcal{C}_c(R_*,\delta_*)$, let  $\gamma_1,\gamma_2$ be two $c$-parabolic motions with 
    initial configuration $x_0$, write $Z_i(\tau) = \big(\rho_i(\tau), s_i(\tau), v_i(\tau), w_i(\tau)\big)~ (i=1,2)$ be their McGehee lifts.
    One obtains
\begin{equation*}
        \pi(Z_1(0)) = (\rho_1(0),s_1(0)) = (\rho_2(0),s_2(0)) = \pi(Z_2(0)),
\end{equation*}
thus by injectivity, $Z_1(0) = Z_2(0)$.

The physical velocities are given by 
\begin{equation*}
        \dot{x} = \dot{r}s + r\dot{s} = r^{-\frac{1}{2}}  (vs +w) = \rho(vs+w),
\end{equation*}
thus $\dot{\gamma}_1(0) = \dot{\gamma}_2(0)$.
Therefore, by the uniqueness of solutions of the ODE, 
\begin{equation*}
        \gamma_1(t) = \gamma_2(t),~t\geq0.
\end{equation*}
On the other hand, since $W(\eta)\subset \mathcal{W}$, $v=V(s,w)>0$.
By the graph representation \eqref{eq:rectangular-stable-graph},
\begin{equation*}
        \dot{x} = \rho(vs +w) = \rho\big(sV(s,G(\rho,s)) +G(\rho,s)\big),
\end{equation*}
we define the velocity map
\begin{equation*}
\mathcal V_c(r,s)
=
r^{-1/2}
\left[
V\left(s,G(r^{-1/2},s)\right)s
+
G(r^{-1/2},s)
\right].
\end{equation*}
$V$ and $G$ are jointly $C^\infty$ in $(\rho, s)$, then $\mathcal V_c(r,s)$ is $C^\infty$, which is exactly \eqref{eq:smooth-initial-velocity-map}. 
Therefore, for any $x_0=r_0s_0\in \mathcal{C}_c(R_*,\delta_*)$
 the physical
solution given by $(x_0, \mathcal{V}_c(r_0,s_0))$ is defined for all $t\geq0$, is the unique
$c$-parabolic that remains in $\mathcal C_c(R_*,\delta_*)$.

\end{deferredproof}

\subsection{Cone uniqueness for parabolic geodesic rays}
\label{subscn:parabolic-rays-nested-cone}

We remove the confinement hypothesis by using the radial velocity of an
actual minimizing ray. An action comparison makes $v_*-v(0)$ small.
The monotonicity of $v$ then forces the radius to increase, while the
zero-energy identity bounds the entire remaining shape length.
The elementary excess and comparison estimates below will also be used
in the global uniqueness argument.

Set $B(r):=2v_*\sqrt r$ and, for an absolutely continuous collision-free
curve $x=rs$ on a compact time interval, define
\[
 \mathcal R_{[a,b]}(x)
 :=\int_a^b L(x,\dot x)\,dt-B(r(b))+B(r(a)).
\]
When the endpoints are specified in McGehee time, this notation refers
to the corresponding physical arc.

\begin{lemma}[Renormalized excess and shape tails]
\label{lem:renormalized-endpoint finite-tail-excess shape-tail}
For every absolutely continuous collision-free curve $x=rs$,
\begin{equation}
 \label{eq:renormalized-density}
 \mathcal R_{[a,b]}(x)=\int_a^b
 \left[
 \frac12\left(\dot r-\frac{v_*}{\sqrt r}\right)^2
 +\frac12r^2\|\dot s\|^2
 +\frac{U(s)-U_{min}}{r}
 \right]dt\geq0.
\end{equation}
Along any zero-energy solution,
\begin{equation}
 \label{eq:endpoint-identity}
 \mathcal R_{[\tau_1,\tau_2]}
 =\left[2\sqrt r\,(v-v_*)\right]_{\tau_1}^{\tau_2}.
\end{equation}
For every $c$-parabolic motion, put $F:=v_*-v\geq0$. Its tail excess
is finite and satisfies
\begin{equation}
 \label{eq:total-excess-exact}
 \mathcal R_{[\tau,\infty)}=2\sqrt{r(\tau)}F(\tau)
 \longrightarrow0\qquad(\tau\to\infty).
\end{equation}
If $v(\tau)\geq0$, then
\begin{equation}
 \label{eq:shape-length-from-radial-defect}
 d_{\mathcal E}(s(\tau),c)
 \leq\int_\tau^\infty\|w(\sigma)\|\,d\sigma
 \leq\frac4{\sqrt{v_*}}\sqrt{F(\tau)}.
\end{equation}
In particular, at such times,
\begin{equation}
 \label{eq:shape-tail-estimate}
 d_{\mathcal E}(s(\tau),c)^2\sqrt{r(\tau)}
 \leq\frac8{v_*}\mathcal R_{[\tau,\infty)}.
\end{equation}
Thus $d_{\mathcal E}(s(\tau),c)^2\sqrt{r(\tau)}\to0$ along every
$c$-parabolic motion.
\end{lemma}

\begin{proof}
Since $\langle s,\dot s\rangle=0$ and $U(rs)=r^{-1}U(s)$,
$$
L(x,\dot x)
=
\frac12\dot r^2
+\frac12r^2\lVert\dot s\rVert^2
+\frac{U(s)}{r}.
$$
Moreover, let $B(r) = \sqrt{8U_{min}r} = 2v_*\sqrt{r}$
$$
\frac{d}{dt}B(r)
=
\frac{v_*}{\sqrt r}\dot r.
$$
Completing the square gives
\begin{equation}\label{eq:completion-square-proof}
\begin{aligned}
    L(x,\dot x)-\frac{d}{dt}B(r(t))
    &=\left(\frac{1}{2}\dot{r}^2  - \sqrt{\frac{2U_{min}}{r}}\dot{r} + \frac{U_{min}}{r}\right) 
     + \frac{1}{2}r^2 \|\dot{s}\|^2 + \frac{U(s) - U_{min}}{r} \\
    &= \frac{1}{2}\left(\dot{r} - \sqrt{\frac{2U_{min}}{r}}\right) ^2 + 
     + \frac{1}{2}r^2 \|\dot{s}\|^2 + \frac{U(s) - U_{min}}{r}.
\end{aligned}
\end{equation}
Integration proves \eqref{eq:renormalized-density}. Its integrand is
nonnegative because $U(s)\geq U_{min}$ on $\mathcal E$.

Now assume that $\gamma$ is a zero-energy solution. Using
$dt=r^{3/2}d\tau$, $\dot r=r^{-1/2}v$, and
$\dot s=r^{-3/2}w$, formula \eqref{eq:renormalized-density} becomes
\begin{equation}
\label{eq:renormalized-tau-density}
\mathcal R_{[\tau_1,\tau_2]}(\gamma)
=
\int_{\tau_1}^{\tau_2}
\sqrt r
\left[
\frac12(v-v_*)^2
+\frac12\lVert w\rVert^2
+U(s)-U_{min}
\right]d\tau.
\end{equation}
The zero-energy relation and $v_*^2=2U_{min}$ give us
\begin{equation*}
    U(s) - U_{min}  = \frac{1}{2} \left(v^2 + \|w\|^2 - v_*^2\right),
\end{equation*}
which 
reduces the integrand to
$$
\sqrt r\bigl(v^2-v_*v+\lVert w\rVert^2\bigr).
$$
On the other hand, $r'=rv$ and
$v'=\dfrac12\lVert w\rVert^2$ yield
\begin{equation}
\label{eq:endpoint-derivative-proof}
\begin{aligned}
    \frac{d}{d\tau}\left(2\sqrt r\,(v-v_*)\right)
    &=2\left[\frac{1}{2}\sqrt{r} v(v-v_*) + \sqrt{r}v'\right]\\
    &=\sqrt r\bigl(v^2-v_*v+\lVert w\rVert^2\bigr).
\end{aligned}
\end{equation}
Integrating \eqref{eq:endpoint-derivative-proof} proves
\eqref{eq:endpoint-identity}.

For a $c$-parabolic motion, \Cref{prop: parabolic motions hit ea+ at infty} gives us $F\geq0$ and $\displaystyle\lim_{\tau\to+\infty}F = 0$.
Set $E:=2\sqrt r\,F$. The zero-energy identity gives
\[
 \|w\|^2=2(U(s)-U_{min})+(v_*+v)F,
\]
and hence the exact differential identity
\begin{equation}
 \label{eq:radial-defect-excess-decay}
 E'=\sqrt r\,(vF-\|w\|^2)
 =-\frac{v_*}{2}E-2\sqrt r\,(U(s)-U_{min})
 \leq-\frac{v_*}{2}E.
\end{equation}
Therefore $E\to0$. Letting the terminal time tend to infinity in
\eqref{eq:endpoint-identity} proves \eqref{eq:total-excess-exact}.

If $v(\tau)\geq0$, then $v(\sigma)\geq0$ for $\sigma\geq\tau$, and
$\|w\|^2\geq v_*F$ on that tail. Wherever $F>0$, we have
\[
 -\frac{d}{d\sigma}\sqrt F
 =\frac{\|w\|^2}{4\sqrt F}
 \geq\frac{\sqrt{v_*}}4\|w\|.
\]
Integration, together with $F\to0$ and $s\to c$, gives
\eqref{eq:shape-length-from-radial-defect}.
If $F$ reaches zero, monotonicity and $v\to v_*$ imply that
$F=0$, $w=0$ and $U(s)=U_{min}$ by the zero energy constraint thereafter, 
so $s=c$ by $s\to c$, the same conclusion holds.
Squaring \eqref{eq:shape-length-from-radial-defect} and using
\eqref{eq:total-excess-exact} yields \eqref{eq:shape-tail-estimate}.
Finally, $v\to v_*>0$, so this estimate applies on every sufficiently
late tail and proves the last assertion.
\end{proof}

\begin{lemma}[Shell adjustment]
\label{lem:shell-adjustment}

There are $\delta_{\mathrm{sh}}>0$ and $C_{\mathrm{sh}}>0$ such that,
for every $b>0$ and $d_{\mathcal E}(s,c)<\delta_{\mathrm{sh}}$, there
are collision-free free-time comparison curves from $bs$ to $2bc$,
and from $bc$ to $2bs$, each with excess at most
\[
 C_{\mathrm{sh}}d_{\mathcal E}(s,c)^2\sqrt b.
\]
\end{lemma}

\begin{proof}
Since $c$ is a critical point of $U|_{\mathcal E}$ and $U$ is smooth there,
choose $\delta_{\mathrm{sh}}\in(0,\delta_0)$ so that
there is a constant $C$ depending on $U$ and $\delta_{\mathrm{sh}}$ such that
\[
 0\leq U(s)-U_{min}\leq C\,d_{\mathcal E}(s,c)^2,\quad s\in B_\mathcal{E}(c,\delta_\mathrm{sh}).
\]
Write $\theta=d_{\mathcal E}(s,c)$. 
Given any $s\in B_\mathcal{E}(c,\delta_{\mathrm{sh}})$, let $\sigma:[0,1]\to\mathcal E$ be
 the constant-speed minimizing geodesic
from $s$ to $c$. Thus
$\sigma(0)=s,
\sigma(1)=c,
\lVert\sigma'(u)\rVert=\theta.$
For $b\leq r\leq2b$, set
$$
s(r)
=
\sigma\left(\frac{r-b}{b}\right).
$$
Then
$$
\left\lVert\frac{ds}{dr}\right\rVert
=
\frac{\theta}{b},
\qquad
d_{\mathcal E}(s(r),c)\leq\theta.
$$

Parametrize the comparison curve by $\dot r=v_*/\sqrt r$.
The physical duration for the parameter \(r\) ranging from \(b\) to \(2b\) should be
\(
\displaystyle T_b = \int_{b}^{2b} \frac{\sqrt{r}}{v_*}dr = \frac{2}{3 v_*} \left[ (2b)^{3/2} - b^{3/2} \right].
\)
The radial square in \eqref{eq:renormalized-density} vanishes, 
Moreover,
$$
dt=\frac{\sqrt r}{v_*}\,dr,
\qquad
\dot s
=
\frac{v_*}{\sqrt r}\frac{ds}{dr}.
$$

The shape contribution is therefore
\begin{equation}
\label{eq:shell-angular-cost}
\begin{aligned}
\int_0^{T_b}\frac12r^2\lVert\dot s\rVert^2\,dt
&=
\frac{v_*}{2}
\int_b^{2b}
r^{3/2}
\left\lVert\frac{ds}{dr}\right\rVert^2\,dr\\
&\leq
\frac{v_*}{2} \frac{\theta^2}{b^2}
\int_b^{2b}r^{3/2}\,dr\\
&=
C_1\theta^2 \sqrt b.
\end{aligned}
\end{equation}
The potential contribution satisfies
\begin{equation}
\label{eq:shell-potential-cost}
\begin{aligned}
\int_0^{T_b}\frac{U(s)-U_{min}}{r}\,dt
&=
\frac{1}{v_*}
\int_b^{2b}
\bigl(U(s(r))-U_{min}\bigr)r^{-1/2}\,dr\\
&\leq
\frac{C}{v_*}\theta^2
\int_b^{2b}r^{-1/2}\,dr\\
&=
C_2\theta^2\sqrt b.
\end{aligned}
\end{equation}
Adding \eqref{eq:shell-angular-cost} and
\eqref{eq:shell-potential-cost} proves the first assertion.

Reversing the shape interpolation while keeping the radial motion
outward proves the second assertion.
\end{proof}

A geodesic ray
minimizes the total branch excess among all stable $c$-parabolic solutions
with the same initial configuration.

\begin{proposition}[Global excess selection by free-time minimality]
\label{prop:global-excess-selection}
Let $\gamma$ be a $c$-parabolic geodesic ray, and let $\eta$ be a
$c$-parabolic solution with the same initial configuration.   Then
\begin{equation}
\label{eq:global-excess-selection}
\mathcal R_{[0,\infty)}(\gamma)
\leq
\mathcal R_{[0,\infty)}(\eta).
\end{equation}
In particular, if both $\gamma$ and $\eta$ are $c$-parabolic geodesic
rays, then their total excesses are equal.
\end{proposition}

\begin{proof}
Choose physical times $T_n\to+\infty$ along $\eta$ and write
$$
y_n:=\eta(T_n)=r_n s_n.
$$
Since $\eta$ is $c$-parabolic,
$
r_n\longrightarrow+\infty.
$
Choose physical times $S_n\to+\infty$ along $\gamma$ and write
$$
z_n:=\gamma(S_n)=R_n q_n,
$$
with
$
R_n\geq4r_n.
$

For all sufficiently large $n$, the endpoint shapes lie in the shell
ball $B_\mathcal{E}(c,\delta_{\mathrm{sh}})$.  We give a curve $\zeta_n$ joining $\eta(T_n)$ to $\gamma(S_n)$ by three segments.
The first segment is a shell adjustment from $r_ns_n$ to $2r_n c$. The second segment follows the homothetic parabolic motion
$s\equiv c,~\dot\varrho=\dfrac{v_*}{\sqrt\varrho}$, from $2r_nc$ to $(R_n/2)c$, it is outward because $R_n\geq4r_n$, and all three
terms in the excess vanish. The third one is again a shell adjustment from $R_n/2 c$ to $R_n q_n$.
Therefore, \Cref{lem:shell-adjustment,lem:renormalized-endpoint finite-tail-excess shape-tail} gives 
\begin{equation*}
\mathcal R(\zeta_n)
\leq
C_{\mathrm{sh}}
\left[
d_{\mathcal E}(s_n,c)^2\sqrt r_n
+
d_{\mathcal E}(q_n,c)^2\sqrt R_n
\right]
\end{equation*}
such that $\mathcal R(\zeta_n)\longrightarrow0$.

The concatenation
$
\eta|_{[0,T_n]}*\zeta_n
$
connects the common initial configuration to
$z_n=\gamma(S_n)$.  Since $\gamma$ is a free-time minimizer,
\begin{equation*}
A_L(\gamma|_{[0,S_n]})
\leq
A_L(\eta|_{[0,T_n]})
+
A_L(\zeta_n).
\end{equation*}
Let $r_0$ be the common initial radius.  Subtracting
$
B(R_n)-B(r_0)
$
from both sides and splitting the right-hand boundary terms yields
\begin{equation}
\label{eq:selection-truncated}
\mathcal R_{[0,S_n]}(\gamma)
\leq
\mathcal R_{[0,T_n]}(\eta)
+
\mathcal R(\zeta_n).
\end{equation}
The finite excess is monotone in its terminal time because its density is
nonnegative.  Letting $n\to\infty$ in
\eqref{eq:selection-truncated}  proves
\eqref{eq:global-excess-selection}.  If $\eta$ is also a geodesic ray,
interchanging the two rays gives the reverse inequality.
\end{proof}

\begin{proposition}[Radial monotonicity and configuration trapping]
\label{prop:configuration-trapping}
Fixing the constant \(\delta_{\mathrm{sh}}\) and \(C_{\mathrm{sh}}\) from \Cref{lem:shell-adjustment}.
There is $K>0$ such that every $c$-parabolic geodesic ray with
$\gamma(0)=r_0s_0$ and 
$\theta_0:=d_{\mathcal E}(s_0,c)<\min\left\{\delta_{\mathrm{sh}},\sqrt{v_*/C_{\mathrm{sh}}}\right\}$ satisfies
\[
v(\tau)\geq\frac{v_*}{2},\quad
r(t)\geq r_0,\quad
d_{\mathcal E}(s(t),c)\leq K\theta_0
\quad(t,\tau\geq0).
\]
Consequently, for every $0<\delta\leq\min\{\delta_{\mathrm{sh}},1\}$,
there exists $\varepsilon\in(0,\delta/2)$ such that every such ray
with $\theta_0<\varepsilon$ remains in
$\mathcal C_c(r_0/2,\delta)$ for all $t\geq0$.
\end{proposition}

\begin{proof}
For sufficiently large physical time $t_0$ such that $r(T)>4r_0$ and $d_\mathcal{E}(s(T), c)<\delta_{\mathrm{sh}}$ for $T>t_0$, join $r_0s_0$ to
$\gamma(T)=r(T)s(T)$ by a shell adjustment from $r_0s_0$ to $2r_0c$, an outward
homothetic segment from $2r_0c$ to $(r(T)/2)c$, and a terminal shell adjustment
from $(r(T)/2)c$ to $r(T)s(T)$.
This is possible since $r(T)>4r_0$ and $s(T)$ lies in the $\delta_{\mathrm{sh}}$-shell ball.
The middle segment has zero excess. The excess estimate in \Cref{lem:shell-adjustment}, free-time excess minimality and the
cancellation of the endpoint corrections give
\[ 
 \mathcal R_{[0,T]}(\gamma)
 \leq C_{\mathrm{sh}}\left[
 \theta_0^2\sqrt{r_0}
 +d_{\mathcal E}(s(T),c)^2\sqrt{r(T)}
 \right].
\]
The terminal term tends to zero by \Cref{lem:renormalized-endpoint finite-tail-excess shape-tail}. Passing to the
limit and using the exact excess formula yields
\begin{equation}
 \label{eq:uniform-excess-bound}
 2\sqrt{r_0}\,(v_*-v(0))
 =\mathcal R_{[0,\infty)}(\gamma)
 \leq C_{\mathrm{sh}}\theta_0^2\sqrt{r_0}.
\end{equation}
Thus
\[
 0\leq F(0)=v_*-v(0)\leq\frac12C_{\mathrm{sh}}\theta_0^2,
\]
 which gives $v(0)>v_*/2$.
Since $v'\geq0$, the same lower bound holds for all subsequent times,
so $r'=rv>0$ and $r(t)\geq r_0$.
Moreover, $F(\tau)\leq F(0)$ and
\eqref{eq:shape-length-from-radial-defect} give
\[
 d_{\mathcal E}(s(\tau),c)
 \leq\frac4{\sqrt{v_*}}\sqrt{F(\tau)}
 \leq\sqrt{\frac{8C_{\mathrm{sh}}}{v_*}}\,\theta_0.
\]
Take $K:=\max\{1,\sqrt{8C_{\mathrm{sh}}/v_*}\}$ and choose
\[
 0<\varepsilon<\min\left\{
 \frac{\delta}{2},\ \sqrt{\frac{v_*}{C_{\mathrm{sh}}}},\
 \frac{\delta}{K}\right\}.
\]
Then $r(t)\geq r_0>r_0/2$ and
$d_{\mathcal E}(s(t),c)\leq K\theta_0<K\varepsilon<\delta$,
which proves the trapping assertion.
\end{proof}

We now prove \Cref{thm:uniqueness-cone-parabolic-geodesic-ray}.
Keep $R_*,\delta_*$ from
\Cref{thm:uniqueness-cone-confined-parabolic-motion}, and set
$\bar\delta:=\min\{\delta_*,\delta_{\mathrm{sh}},1\}$.
Choose $\varepsilon\in(0,\bar\delta/2)$ from
\Cref{prop:configuration-trapping}. If $r_0>2R_*$ and
$d_{\mathcal E}(s_0,c)<\varepsilon$, every $c$-parabolic geodesic ray
from $r_0s_0$ remains in $\mathcal C_c(r_0/2,\bar\delta)$,
hence in $\mathcal C_c(R_*,\delta_*)$.
The cone-confined uniqueness theorem therefore gives uniqueness,
without imposing a confinement hypothesis on the ray.

\section{Generic Uniqueness of Parabolic Geodesic Rays}
\label{sec:global-generic-uniqueness}

The aim of this section is to pass from the cone uniqueness established near
the outgoing parabolic equilibrium
to a global uniqueness statement valid for almost every collision-free
initial configuration.  The proof is entirely phase-space based.  Every
$c$-directed parabolic geodesic ray is first captured by the local strong
stable manifold; the local manifold is then propagated backwards by the
McGehee flow.  Sard's theorem removes the critical values of the resulting
configuration projections, while the remaining multiple regular sheets are
separated by their total renormalized excess.  Free-time minimality implies
that two geodesic rays starting from the same configuration have the same
branch excess, and the differential of the branch excess shows that such
ties lie on codimension-one hypersurfaces.

We adopt the same notation as in the last section.
Throughout this section, unless otherwise stated, Let $\mathcal{N}^+:=\mathcal{N}^+\cap\{\rho>0\}$ and use $Y = (\rho, s)$ to denote points in $\mathcal{N}^+$.
Note that the map 
$$
\Theta:\mathcal{N}^+\longrightarrow\Omega,
\qquad
\Theta(\rho,s)=\rho^{-2}s,
$$
is a smooth diffeomorphism. 
We now propagate the physical part of the local strong stable manifold
backwards.  The notation is chosen to agree with the flow domains already
introduced earlier. Fix a $\delta\in(0,\delta_0)$, set
\begin{equation*}
W_0
:=
W_{loc}^{ss}(e_c^+;\delta)
\cap\{\rho>0\}.
\end{equation*}
For every integer $n\geq 1$, define
\begin{equation*}
W_n
:=
\Phi^{-n}
\bigl(
W_0\cap\mathcal D_n^-
\bigr).
\end{equation*}

\begin{proposition}[Countable backward stable sheets]
\label{prop:countable-backward-stable-sheets}
Each $W_n$ is an $m$-dimensional $C^\infty$ embedded submanifold in $\mathcal P$. The initial McGehee phase point of every
$c$-parabolic geodesic ray belongs to
$$
\bigcup_{n=0}^{\infty}W_n.
$$
\end{proposition}

\begin{proof}
 By the previously established flow-domain
result, $\mathcal D_n^-$ is open in $\mathcal{P}$ and
$$
\Phi^{-n}:\mathcal D_n^-\longrightarrow\mathcal D_n^+
$$
is a $C^\infty$ diffeomorphism.  Thus
$W_0\cap\mathcal D_n^-$ is an submanifold in $\mathcal{P}$, and is also $m$-dimensional and  $C^\infty$. The 
image of $W_0\cap\mathcal D_n^-$ under $\Phi^{-n}$ is again an $m$-dimensional embedded submanifold in $\mathcal{P}$.
This proves the first assertion.

Let $Z_0$ be the initial phase point of a $c$-parabolic geodesic ray.
There is
$\tau_*\geq0$ such that
$$
Y_*:=\Phi^{\tau_*}(Z_0)\in W_0.
$$
Choose an integer $n\geq\tau_*$ and put
$$
Y_n:=\Phi^{n-\tau_*}(Y_*).
$$
The local strong stable manifold is positively invariant, so
$Y_n\in W_0$.  The backward orbit of $Y_n$ exists for time $n$ and
returns to $Z_0$, hence $Y_n\in\mathcal D_n^-$ and
$$
Z_0=\Phi^{-n}(Y_n)\in W_n.
$$
\end{proof}

The next proposition applies Sard's theorem separately to the countably
many backward sheets.  This avoids treating their global union as a single
embedded submanifold.

\begin{proposition}[Critical values and regular sheets]
\label{prop:regular-sheet-decomposition}
For $n\geq0$, let
$$
\pi_n:=\pi|_{W_n}:W_n\longrightarrow\mathcal{N}^+.
$$
There is a null set $\mathscr C_{\mathrm{crit}}\subset\mathcal{N}^+$ and a
countable family of open sets
$$
\mathcal V_i\subset W_{n(i)},
\qquad
U_i\subset\mathcal{N}^+,
\qquad
i\in\mathbb N,
$$
such that
$$
\pi_{n(i)}|_{\mathcal V_i}:\mathcal V_i\longrightarrow U_i
$$
is a $C^\infty$ diffeomorphism and the sets $\mathcal V_i$ cover all
regular points of all maps $\pi_n$, we call $U_i$ regular sheet.
Each sheet has a smooth section
\begin{equation}
\label{eq:regular-sheet-section}
Z_i:U_i\longrightarrow\mathcal V_i,
\qquad
Z_i(Y)=\bigl(\rho,s,v_i(Y),w_i(Y)\bigr),
\qquad
Y=(\rho,s).
\end{equation}
\end{proposition}

\begin{proof}
Since $\dim W_n=m=\dim\mathcal{N}^+$,
a point $Z\in W_n$ is regular precisely when
$$
d\pi_n(Z):T_ZW_n\longrightarrow T_{\pi(Z)}\mathcal{N}^+
$$
is a linear isomorphism. 
Write $\mathrm{Crit}(\pi_n)$ as the set of critical points of $\pi_n$ in $W_n$.
By Sard's theorem, the critical-value set
$$
\pi_n\bigl(\Crit(\pi_n)\bigr)
$$
has measure zero in $\mathcal{N}^+$.  Therefore
$$
\mathscr C_{\mathrm{crit}}
:=
\bigcup_{n=0}^{\infty}
\pi_n\bigl(\Crit(\pi_n)\bigr) \subset\mathcal{N}^+.
$$
is a countable union of null sets and is itself null.

At every regular point, the inverse function theorem gives a neighborhood
on which $\pi_n$ is a diffeomorphism onto an open subset of $\mathcal{N}^+$.
Each $W_n$ is second countable, so its regular part admits a countable
subcover by such neighborhoods.  Taking the union over $n\in\mathbb N$
produces the stated countable family.  Formula
\eqref{eq:regular-sheet-section} is the inverse of the corresponding
restricted projection.
\end{proof}

Before comparing different sheets, we attach a scalar excess to each
branch.  The following statement applies to all stable $c$-parabolic
solutions on a sheet; no minimizing property is used.

\begin{lemma}[Branch excess and its endpoint formula]
\label{lem:branch-excess-endpoint}
For $Y=(\rho,s)\in U_i$, let $\gamma_i^Y$ be the physical solution starting from
$Z_i(Y)$.  Define
\begin{equation}
\label{eq:branch-excess-definition}
\mathcal{R}_i(Y)
:=
\mathcal R_{[0,+\infty)}(\gamma_i^Y)
\end{equation}
exists, is finite, and satisfies
\begin{equation}
\label{eq:branch-excess-explicit}
\mathcal{R}_i(Y)
=
2\rho^{-1}\bigl(v_*-v_i(Y)\bigr).
\end{equation}
In particular, $\mathcal{R}_i$ is $C^\infty$ on $U_i$.
\end{lemma}

\begin{proof}
Every point of $W_{n(i)}$ flows after finite McGehee time into $W_0$.
Consequently, $\gamma_i^Y$ is a collision-free zero-energy
$c$-parabolic solution, although it is not assumed here to be a geodesic
ray.
Along any zero-energy $c$-parabolic solution, \Cref{lem:renormalized-endpoint finite-tail-excess shape-tail}
 yields
$$
\mathcal{R}_i(Y)
=
2\sqrt{r(0)}\bigl(v_*-v(0)\bigr).
$$
For $Z_i(Y)=(\rho,s,v_i(Y),w_i(Y))$, one has
$\sqrt{r(0)}=\rho^{-1}$ and $v_i(Y) = v(0)$, which proves
\eqref{eq:branch-excess-explicit}.  Smoothness follows from the smoothness
of the regular-sheet section.
\end{proof}

The next proposition gives the differential of the renormalized excess on a
regular sheet in terms of the McGehee phase variables themselves. The
crucial point is that the derivative of the radial velocity component
$v_i$ can be eliminated by the linearized McGehee equations and the
strong-stable decay at infinity. 
The differential will later allow us to show that the coincidence set of the
renormalized excess associated with two distinct regular sheets is locally
a hypersurface.
\begin{proposition}[Differential of a branch excess]
\label{prop:branch-excess-differential}
For every regular sheet,
\begin{equation}
\label{eq:branch-excess-differential}
d\mathcal{R}_i
=
2\rho^{-2}(v_i-v_*) d\rho
-
\rho^{-1}\langle w_i, d s\rangle.
\end{equation}
Consequently, on $U_i\cap U_j$,
\begin{equation}
\label{eq:branch-difference-differential}
d(\mathcal{R}_i-\mathcal{R}_j)
=
2\rho^{-2}(v_i-v_j) d\rho
-
\rho^{-1}\langle w_i-w_j, d s\rangle.
\end{equation}
If $Z_i(Y)\neq Z_j(Y)$, then
$$
d(\mathcal{R}_i-\mathcal{R}_j)(Y)\neq0.
$$
\end{proposition}
See \proofref{prop:branch-excess-differential}{apdx:3}.

We can now isolate the only remaining regular multiple-branch set.

\begin{proposition}[Tie sets are null]
\label{prop:tie-sets-null}
For two regular sheets, define
\begin{equation}
\label{eq:tie-set-definition}
\mathscr Z_{ij}
:=
\left\{
Y\in U_i\cap U_j:
\mathcal{R}_i(Y)=\mathcal{R}_j(Y),
\quad
Z_i(Y)\neq Z_j(Y)
\right\}.
\end{equation}
Then $\mathscr Z_{ij}$ has measure zero in $\mathcal{N}^+$.  Consequently,
$$
\mathscr Z:=\bigcup_{i,j\in\mathbb N}\mathscr Z_{ij}
$$
is a null set.
\end{proposition}

\begin{proof}
At every $Y\in\mathscr Z_{ij}$,
Proposition~\ref{prop:branch-excess-differential} gives
$$
d(\mathcal{R}_i-\mathcal{R}_j)(Y)\neq0.
$$
The implicit function theorem therefore shows that, near $Y$, the level
set
$$
\mathcal{R}_i-\mathcal{R}_j=0
$$
is a $C^1$ hypersurface of codimension one.  Such a hypersurface has
measure zero in local coordinates.  Since $\mathcal{N}^+$ is second
countable, $\mathscr Z_{ij}$ is covered by countably many such coordinate
neighborhoods and is null.  The family of sheet pairs is countable, so
their union is also null.
\end{proof}

We finish here the proof of \Cref{thm:global-generic-uniqueness} by combining eventual capture, Sard's theorem, and the
variational selection principle.
\begin{proof}
Set
$$
\mathscr N_c
:=
\mathscr C_{\mathrm{crit}}\cup\mathscr Z.
$$
By Propositions~\ref{prop:regular-sheet-decomposition} and
\ref{prop:tie-sets-null}, this is a null set in $\mathcal{N}^+$.

Take
$$
Y=(\rho,s)\in\mathcal{N}^+\setminus\mathscr N_c
$$
and put
$$
x=\Theta(Y)=\rho^{-2}s.
$$
Suppose that two $c$-parabolic geodesic rays $\gamma_1,\gamma_2$ start
from $x$.   Since $Y$ is not a critical
value, the two points lie on regular sheets, say
$Z_i(Y), Z_j(Y).$
Applying \Cref{prop:global-excess-selection} in both
directions gives
$$
\mathcal{R}_i(Y)=\mathcal{R}_j(Y).
$$
If $Z_i(Y)\neq Z_j(Y)$, then
$$
Y\in\mathscr Z_{ij}\subset\mathscr N_c,
$$
contrary to the choice of $Y$.  Hence
$$
Z_i(Y)=Z_j(Y).
$$
The two rays therefore have the same initial McGehee velocity components,
and thus the same physical initial velocity
$$
\rho(vs+w).
$$
Uniqueness for the Newtonian initial-value problem implies
$$
\gamma_1=\gamma_2.
$$
Thus multiple rays can occur only over $\mathscr N_c$.

Finally, $\Theta$ and $\Theta^{-1}$ are smooth and locally Lipschitz on a
countable exhaustion by relatively compact coordinate sets.  Hence
$\Theta(\mathscr N_c)$ is a null set in $\Omega$.  By
\eqref{eq:multiple-ray-set},
$$
\mathfrak M_c\subset\Theta(\mathscr N_c),
$$
which proves the theorem.
\end{proof}

\appendix

\section{Proofs concerning the strong stable manifold}\label{apdx:1}

\begin{deferredproof}{lem:Duhamel's Formula}
    Set 
    \[
    z(\tau) = \mathrm{e}^{-A\tau} Y(\tau)
    \]
    and differentiate it, we obtain
    \[
    \begin{aligned}
    z'(\tau) &= -\mathrm{e}^{-A\tau} A Y(\tau) + \mathrm{e}^{-A\tau} Y'(\tau) \\
    &= -\mathrm{e}^{-A\tau} A Y(\tau) + \mathrm{e}^{-A\tau} \bigl(A Y(\tau) +  b(\tau)\bigr) \\
    &= \mathrm{e}^{-A\tau}  b(\tau).
    \end{aligned}
    \]
    Integrating gives
    \[
    z(\tau_2) - z(\tau_1) =  \int_{\tau_1}^{\tau_2} \mathrm{e}^{-A\sigma}  b(\sigma) \, d\sigma.
    \]
    Multiplying  $\mathrm{e}^{A\tau_2}$ yields the final formula
    \[
     Y(\tau_2) = \mathrm{e}^{A(\tau_2 - \tau_1)} Y(\tau_1) + \int_{\tau_1}^{\tau_2} \mathrm{e}^{A(\tau_2-\sigma)}  b(\sigma)  d\sigma. 
    \]
\end{deferredproof}

\begin{deferredproof}{lem:conjucation of flows}
    If we set
\[
Z = \chi^{-1}(Y),
\]
then \eqref{eq:f is the pushforward} becomes
\[
f(Y) = D\chi_Z \, \mathfrak{X}(Z).
\]
Since $\chi$ is a local diffeomorphism,
\[
D\chi_Y^{-1} = \bigl( D\chi_{\chi^{-1}(Y)} \bigr)^{-1}.
\]
Hence we have the needed relation between the original equation and the equation in coordinates $\chi$:
\begin{equation}\label{eq:the coordinate vector field relation}
    \mathfrak{X}(Z) = D\chi_{\chi(Z)}^{-1} \, f(\chi(Z)). 
\end{equation}

Remember $Y_a'(\tau) = f(Y_a(\tau))$, denote $\widehat{Z}_a(\tau) = \chi^{-1}(Y_a(\tau))$
and differentiating it with respect to $\tau$ and using the chain rule gives
\begin{equation*}
    \begin{aligned}
        \frac{d}{d\tau}\widehat{Z}_a(\tau) &= \frac{d}{d\tau}\chi^{-1}(Y_a(\tau))\\
                                           &= D\chi_{Y_a(\tau)}^{-1} \, Y_a'(\tau)\\
                                           &=D\chi_{Y_a(\tau)}^{-1} \, f(Y_a(\tau))
    \end{aligned}
\end{equation*}
Since $Y_a(\tau) = \chi(\widehat{Z}_a(\tau))$, recall \eqref{eq:the coordinate vector field relation} implies
\[
D\chi_{Y_a(\tau)}^{-1} \, f(Y_a(\tau)) = \mathfrak{X}(\widehat{Z}_a(\tau)).
\]
Consequently,
\[
 \widehat{Z}_a'(\tau) = \mathfrak{X}(\widehat{Z}_a(\tau)). 
\]
Thus $\widehat{Z}_a$ is a solution of the original McGehee equations.

On the other hand, since $\widehat{Z}_{a}(0)=\chi^{-1}(Y_a(0))=\chi^{-1}(a+h(a))=Z_{a}$,
Hence $\widehat{Z}_a$ satisfies the original initial value problem
\[
\begin{cases}
Z' = \mathfrak{X}(Z) \\
Z(0) = Z_a
\end{cases},
\]
which already has a solution $\tau\to\Phi^\tau (Z_a)$.
Therefore 
\begin{equation}\label{eq:widehat Z_a(tau) = Phi^tau(Z_a)}
    \widehat{Z}_a(\tau)=\Phi^\tau(Z_a)
\end{equation}
holds on the common interval of existence of the two solutions.
However, $\widehat{Z}_a(\tau) = \chi^{-1}(Y_a(\tau))$ is already defined for all $\tau \geq 0$ 
and satisfies the equation above.
Hence the maximal solution starting from $Z_a$ exists on $[0,\infty)$, 
and coincide with $\widehat{Z}_a$.
Therefore \eqref{eq:widehat Z_a(tau) = Phi^tau(Z_a)} holds for all $\tau \geq 0$.
Finally, for $\tau\geq0$
\begin{equation*}
\begin{aligned}
    \chi(\Phi^\tau(Z_a))&=\chi(\widehat{Z}_a(\tau))\\
    &=\chi\left(\chi^{-1}(Y_a(\tau))\right)\\&=Y_a(\tau),
\end{aligned}
\end{equation*}
as claimed.

\end{deferredproof}

\begin{deferredproof}{lem:D Phi^tau = e^{Atau}}
    Choose any $\eta_0\in T_{e_c^+}\mathcal{P}$, let $Z_\epsilon(0)$ be the local smooth curve
    in $\mathcal{P}$ such that 
    \begin{equation*}
        Z_0(0) = e_c^+, \quad \partial_\epsilon Z_\epsilon (0)|_{\epsilon = 0} = \eta_0.
    \end{equation*}
Write 
\begin{equation*}
    Z_\epsilon(\tau):=\Phi^\tau(Z_\epsilon(0)), \quad \eta(\tau):= \partial_\epsilon Z_\epsilon (\tau)|_{\epsilon = 0},
\end{equation*}
then 
\begin{equation*}
    Z_0(\tau)=\Phi^\tau(Z_0(0))=\Phi^\tau(e_c^+) = e_c^+,\quad \eta(0) = \eta_0,
\end{equation*}
and
\begin{equation}\label{eq:eta tau = D Phi^tau eta_0}
    \begin{aligned}
            \eta(\tau) &= \partial_\epsilon Z_\epsilon (\tau)|_{\epsilon = 0} = \partial_\epsilon\Phi^\tau(Z_\epsilon(0))|_{\epsilon = 0} \\
            &=D\Phi^\tau_{Z_\epsilon(0)}\partial_\epsilon Z_\epsilon(0)|_{\epsilon=0} = D\Phi^\tau_{e_c^+}\eta_0.
    \end{aligned}
\end{equation}
On the other hand, $Z_\epsilon'(\tau) = \mathfrak{X}(Z_\epsilon(\tau))$. We differentiate $Z_\epsilon(\tau)$:
\begin{align*}
    \eta'(\tau) & = \partial_\tau\partial_\epsilon Z_\epsilon (\tau)|_{\epsilon =0} 
    = \partial_\epsilon Z_\epsilon'(\tau)|_{\epsilon=0}
    = \partial_\epsilon\mathfrak{X}(Z_\epsilon(\tau))|_{\epsilon=0}\\
    & = D\mathfrak{X}(Z_\epsilon(\tau))\partial_\epsilon Z_\epsilon(\tau)|_{\epsilon=0}
      = D\mathfrak{X}(Z_0(\tau))\eta(\tau)
      = D\mathfrak{X}(e_c^+)\eta(\tau)\\
    &= A\eta(\tau).
\end{align*}
We have obtained the initial value problem of an ODE:
\begin{equation*}
\left\{
\begin{aligned}
\eta'(\tau) &= A\eta(\tau) \\
\eta(0) &= \eta_0 \\
\end{aligned}
\right.,
\end{equation*}
which has a unique solution $\eta(\tau) = \mathrm{e}^{A\tau}\eta_0$.
Combining \eqref{eq:eta tau = D Phi^tau eta_0} yields
\begin{equation*}
    D\Phi^\tau_{e_c^+} = \mathrm{e}^{A\tau}.
\end{equation*}

\end{deferredproof}

\section{Proofs concerning global uniqueness}\label{apdx:3}

\begin{lemma}[Decay of tangent variations on a stable sheet]
\label{lem:stable-sheet-tangent-decay}
Fix $\beta\in(v_*/4,\gamma_\delta)$.  Let
$\mathcal V\subset W_n$ be a $C^1$ sheet and let
$\epsilon\mapsto Z_\epsilon(0)\in\mathcal V$ be a $C^1$ curve.  Write
$Z_\epsilon(\tau) = \Phi^\tau\left(Z_\epsilon(0)\right) = (\rho_\epsilon(\tau),s_\epsilon(\tau),v_\epsilon(\tau),w_\epsilon(\tau))$ and 
$Z_0(\tau)=(\rho(\tau),s(\tau),v(\tau),w(\tau))$. Set
\[
 \eta(\tau)=\partial_\epsilon|_{\epsilon=0}s_\epsilon(\tau),
 \qquad
 \omega(\tau)=\nabla_{\partial\epsilon}|_{\epsilon=0}w_\epsilon(\tau).
\]
There are $T$ and $C$ (depending on the curve and the sheet, but not on
$\tau$) such that
\[
 d_{\mathcal E}(s(\tau),c)+\|w(\tau)\|+
 \|\eta(\tau)\|+\|\omega(\tau)\|
 \le C e^{-\beta\tau},
 \qquad \tau\ge T.
\]
\end{lemma}

\begin{proof}
 
Since $Z_\epsilon(0)\in\mathcal{V}\subset W_n \subset\Phi^{-n}\Bigl[\mathcal O_{T_*}(\delta)\cap\Phi^{-T_*}\bigl(\mathcal U_R^{ss}\cap\mathcal D_{T_*}^-\bigr)\cap\mathcal{D}_n^-\Bigr]$,
one  obtains $Z_\epsilon(T_*+n) = \Phi^{T_*+n}(Z_\epsilon(0))\in\mathcal{U}^{ss}_R$.
Hence the Lyapunov--Perron parametrization \eqref{eq:parametrization of U_R^ss}  gives
$\chi\left(Z_\epsilon(T_*+n) \right) = p(\epsilon) + h(p(\epsilon))$ for some $p(\epsilon)\in B_{\varrho(\beta,R)}$.
In particular, $p(\epsilon)=P_{ss}\chi\bigl(Z_\epsilon(T)\bigr)$ is $C^1$.

According to \Cref{lem:conjucation of flows},
\[
 \chi\bigl(Z_\epsilon(T_*+n+\sigma)\bigr) = 
 \chi\big[\Phi^\sigma(Z_\epsilon(T_*+n))\big] =
 Y_{p(\epsilon)}(\sigma),~~ \sigma\ge0.
\]
The map $p\mapsto Y_p$ is $C^1$ with values in $\mathscr X_\beta$ and so is $\epsilon\to Y_{p(\epsilon)}$; hence both
$Y_{p(0)}(\sigma) = \chi\big(\rho(T_*+n+\sigma),s(T_*+n+\sigma),v(T_*+n+\sigma),w(T_*+n+\sigma)\big)$ and its $\epsilon$-derivative are $O(e^{-\beta\sigma})$.
The same estimate for the $(s,w)$ components and their variations follows
from the bounded differential of the smooth map $\chi^{-1}$ on a smaller
coordinate neighborhood.  Absorbing the fixed entry time $T:=T_*+n$ into $C$ proves the
claim.
\end{proof}

\begin{deferredproof}{prop:branch-excess-differential}
Fix $Y=(\rho,s)\in U_i$ and $h=(a,\zeta)\in T_Y\mathcal N$.  Choose a
$C^1$ curve $Y_\epsilon$ in $U_i$ with $Y_0=Y$ and $Y_0'=h$, and set
\[
 Z_\epsilon(\tau)=\Phi^\tau\bigl(Z_i(Y_\epsilon)\bigr) = (\rho_\epsilon(\tau),s_\epsilon(\tau),v_\epsilon(\tau),w_\epsilon(\tau)),
 \quad Z_0(\tau)=(\rho(\tau),s(\tau),v(\tau),w(\tau)).
\]
In particular, $Z_0(0) = Z_i(Y)$ and $w(0) = w(Z_i(Y)) = w_i(Y)$ by \eqref{eq:regular-sheet-section}.
Let
\[
 \eta=\partial_\epsilon|_{\epsilon=0}s_\epsilon,
 \qquad
 b=\partial_\epsilon|_{\epsilon=0} v_\epsilon,
 \qquad
 \omega=\nabla_{\partial\epsilon}|_{\epsilon=0}w_\epsilon.
\]
Then $\eta(0)=\zeta$ and
$b(0)=dv_i(Y)[a,\zeta]$.  Torsion-freeness gives
$\nabla_\tau\eta=\omega$, while differentiation of $b$ with respect to $\tau$ and
of the zero-energy relation with respect to $\epsilon$ at $\epsilon = 0$ gives
\[
 b'=\langle w,\omega\rangle,
 \qquad
 vb+\langle w,\omega\rangle
 =\langle\nabla_{\mathcal E}U(s),\eta\rangle.
\]
Define
\[
 Q:=2b-\langle w,\eta\rangle.
\]
Using  the preceding identities above,
\begin{equation*}
\begin{aligned}
Q'
&=
2b'
-
\frac{d}{d\tau}
\langle w,\eta\rangle\\
&=
2\langle w,\omega\rangle
-
\left(
\langle\nabla_\tau w,\eta\rangle
+
\langle w,\nabla_\tau\eta\rangle
\right)\\
&=
2\langle w,\omega\rangle
-
\left\langle
\nabla_{\mathcal E}U(s)-\frac12vw,
\eta
\right\rangle
-
\langle w,\omega\rangle\\
&=
\langle w,\omega\rangle
-
\left\langle
\nabla_{\mathcal E}U(s),\eta
\right\rangle
+
\frac12v\langle w,\eta\rangle.
\end{aligned}
\end{equation*}
The differentiated zero-energy relation yields
$$
\langle w,\omega\rangle
-
\left\langle
\nabla_{\mathcal E}U(s),\eta
\right\rangle
=
-vb.
$$
Consequently,
\begin{equation*}
\begin{aligned}
Q'
&=
-vb
+
\frac12v\langle w,\eta\rangle\\
&=
-\frac12v
\left(
2b-\langle w,\eta\rangle
\right).
\end{aligned}
\end{equation*}
Thus
$$
Q'=-\frac12vQ.
$$
Since $\rho'=-\dfrac12v\rho$, we obtain $\left(Q/\rho\right)' = 0$, the quotient $Q/\rho$ is constant along the
variation.

We now use \Cref{lem:stable-sheet-tangent-decay}.  
Since $\nabla_{\mathcal E}U(c)=0$ and $U$ is smooth near $c$, for large $\tau$,
\[
\|\nabla_{\mathcal E}U(s(\tau))\|
\leq C d_{\mathcal E}(s(\tau),c)
\leq C e^{-\beta\tau}.
\]
and
$v(\tau)\ge v_*/2$, so the equality $vb+\langle w,\omega\rangle
 =\langle\nabla_{\mathcal E}U(s),\eta\rangle$ gives
\[
 |b(\tau)|=O(e^{-2\beta\tau}),
 \qquad |Q(\tau)|=O(e^{-2\beta\tau}).
\]
Remember that long a physical stable orbit $v'\ge0$, $v(\tau)\to v_*$ and
$v(\tau)\le v_*$.  Therefore
\[
 \rho(\tau)^{-1}
 =\rho(0)^{-1}\exp\!\left(\frac12\int_0^\tau v(\sigma)\,d\sigma\right)
 \le \rho(0)^{-1}e^{v_*\tau/2}.
\]
Because $\beta>v_*/4$, $Q(\tau)/\rho(\tau)\to0$.  Its constancy then implies
$Q(0)=0$, that is,
\[
 2\,dv_i(Y)[a,\zeta]=\langle w_i(Y),\zeta\rangle.
\]
Finally,
\[
\begin{aligned}
 d\mathcal R_i(Y)[a,\zeta]
 &=2\rho^{-2}(v_i-v_*)a-2\rho^{-1}dv_i(Y)[a,\zeta]\\
 &=2\rho^{-2}(v_i-v_*)a-
\rho^{-1}\langle w_i,\zeta\rangle.
\end{aligned}
\]
Thus
\[
 d\mathcal R_i
 =2\rho^{-2}(v_i-v_*)\,d\rho
  -\rho^{-1}\langle w_i,ds\rangle.
\]
Subtracting the formulas for $i$ and $j$ gives the stated expression for
$d(\mathcal R_i-\mathcal R_j)$.  If this differential vanishes, evaluation on radial tangent
vectors gives $v_i(Y)=v_j(Y)$, while evaluation on all
$\zeta\in T_s\mathcal E$ gives $w_i(Y)=w_j(Y)$.
Together with their common configuration $(\rho,s)$, this implies
$Z_i(Y)=Z_j(Y)$.  Taking the contrapositive proves the last assertion.

\end{deferredproof}

 \bibliographystyle{amsplain}
        \bibliography{references}

\end{document}